\documentclass[12pt]{amsart}
\usepackage{geometry}             

\usepackage{float}
\usepackage{fullpage}
\usepackage[mathscr]{euscript}
\usepackage[all]{xy}
\usepackage{epsfig}
\usepackage[T1]{fontenc}
\usepackage{amsfonts}

\usepackage{graphicx}
\usepackage{amssymb}
\usepackage{mathrsfs}
\usepackage{epstopdf}
\usepackage{url}
\usepackage[dvipsnames,usenames]{xcolor}
\usepackage[colorlinks=true,urlcolor=blue,linkcolor=blue,citecolor=blue]{hyperref}
\usepackage{upgreek}
\hypersetup{
	linkbordercolor={1 0 0}, 
	citebordercolor={0 1 0} 
}
\usepackage{cite}
\usepackage{amsrefs}

\newcommand{\al}{\alpha}
\newcommand{\be}{\beta}
\newcommand{\ga}{\gamma}
\newcommand{\de}{\delta}

\newcommand{\ep}{\varepsilon}

\newcommand{\la}{\lambda}
\renewcommand{\phi}{\varphi}
\newcommand{\si}{\sigma}
\newcommand{\ka}{\kappa}

\newcommand{\Ga}{\Gamma}

\newcommand{\Si}{\Sigma}
\newcommand{\ZZ}{{\mathbb Z}}

\newcommand{\QQ}{{\mathbb Q}}
\newcommand{\RR}{{\mathbb R}}

\newcommand{\RP}{{\mathbb R}{\rm P}}

\newcommand{\cG}{\mathcal G}
\newcommand{\cK}{\mathcal K}
\newcommand{\cQ}{\mathcal Q}
\newcommand{\tr}{\mathsf{T}}
\newcommand{\lto}{\longrightarrow}
\newcommand{\lk}{\operatorname{\ell{\it k}}}
\newcommand{\vlk}{\operatorname{{\it v}\ell{\it k}}}
\newcommand{\nullity}{\operatorname{nul}}
\newcommand{\sig}{\operatorname{sig}}
\newcommand{\sm}{\smallsetminus}
\newcommand{\co}{\colon}
\newcommand{\fr}{{\text{fr}}}

\newcommand{\Int}{{\text{Int}}}

\newcommand{\wh}{\widehat}
\newcommand{\wt}{\widetilde}
\newcommand{\Cyl}{{\text{Cyl}}}

\newcommand*\wbar[1]{
  \hbox{ \kern-0.2em%
    \vbox{%
      \hrule height 0.5pt  
      \kern0.25ex
      \hbox{%
        \kern-0.10em
        \ensuremath{#1}%
        \kern-0.05em
      }%
    }%
  \kern0.05em}%
} 
\makeatletter
\newcommand*\bigcdot{\mathpalette\bigcdot@{.55}}
\newcommand*\bigcdot@[2]{\mathbin{\vcenter{\hbox{\scalebox{#2}{$\m@th#1\bullet$}}}}}
\makeatother

\newtheorem{theorem}{Theorem}[section]
\newtheorem{lemma}[theorem]{Lemma}
\newtheorem{proposition}[theorem]{Proposition}

\newtheorem{corollary}[theorem]{Corollary}

\theoremstyle{definition}     
\newtheorem{definition}[theorem]{Definition}

\theoremstyle{remark}
\newtheorem{remark}[theorem]{Remark}
\newtheorem{example}[theorem]{Example}

\title[The Gordon-Litherland pairing for links in thickened surfaces]{The Gordon-Litherland pairing for \\ links in thickened surfaces}
\author[H. U. Boden]{Hans U. Boden}
\address{Mathematics \& Statistics, McMaster University, Hamilton, Ontario}
\email{boden@mcmaster.ca}

\author[M. Chrisman]{Micah Chrisman}
\address{Mathematics, The Ohio State University, Marion, Ohio}
\email{chrisman.76@osu.edu}

\author[H. Karimi]{Homayun Karimi}
\address{Mathematics \& Statistics, McMaster University, Hamilton, Ontario}
\email{karimih@mcmaster.ca}


\begin{document}

\begin{abstract}
We extend the Gordon-Litherland pairing to links in thickened surfaces, and use it to define signature, determinant, and nullity invariants for links that bound (unoriented) spanning surfaces. The invariants are seen to depend only on the $S^*$-equivalence class of the spanning surface. We prove a duality result relating the invariants from one $S^*$-equivalence class of spanning surfaces to the restricted invariants of the other.

Using Kuperberg's theorem, these invariants give rise to well-defined invariants of checkerboard colorable virtual links. The determinants can be applied to determine the minimal support genus of a checkerboard colorable virtual link. The duality result leads to a simple algorithm for computing the invariants from the Tait graph associated to a checkerboard coloring. We show these invariants simultaneously generalize the combinatorial invariants defined by Im, Lee, and Lee, and those defined by Boden, Chrisman, and Gaudreau for almost classical links. 

We examine the behavior of the invariants under orientation reversal, mirror symmetry, and crossing change. We give a 4-dimensional interpretation of the Gordon-Litherland pairing by relating it to the intersection form on the relative homology of certain double branched covers. This correspondence is made explicit through the use of virtual linking matrices associated to (virtual) spanning surfaces and their associated (virtual) Kirby diagrams.
\end{abstract}

\maketitle
\section*{Introduction}

In \cite{GL-1978}, Gordon and Litherland defined a symmetric, bilinear form for links in $S^3$ together with a choice of unoriented spanning surface. The associated quadratic form was shown to simultaneously generalize the forms of Goeritz and Trotter. They used it to give a simple method to compute the signature for any knot or link from a regular projection. 

It is an open problem to extend the pairing to links in arbitrary 3-manifolds. For instance, Greene has extended it to $\ZZ_2$-homology spheres in the recent paper \cite{greene}. In this paper we will extend the pairing to links in thickened surfaces, and we will use it to define signatures, determinants, and nullities  for links in thickened surfaces and for virtual links.

The general problem of defining signature invariants for links in arbitrary 3-manifolds was studied by Cooper \cite{Cooper} and Mandelbaum--Moishezon \cite{Mandelbaum-Moishezon-1983}. Those papers focus exclusively on links which bound oriented spanning surfaces, namely Seifert surfaces.  A link in a 3-manifold $M$ admits a Seifert surface if and only if it is homologically trivial, and in that case the homology group $H_2(M;\ZZ)$ acts transitively on the set of Seifert surfaces. The resulting signatures depend strongly on the choice of Seifert surface; for instance  when $H_2(M;\ZZ)$ is infinite, there are possibly infinitely many distinct signature invariants  \cite{Cimasoni-Turaev}.

In this paper, we consider the problem of defining signatures for links which bound \textit{unoriented} spanning surfaces. A link in a 3-manifold admits an unoriented spanning surface if and only if it is $\ZZ_2$-homologically trivial, and in the following we focus our attention on  links in thickened surfaces.

Let $\Si$ be a compact, closed, oriented surface, and suppose $L$ is a link in $\Si \times I$ bounding an unoriented spanning surface $F \subset \Si \times I$. Associated to the surface $F$ is a symmetric, bilinear form $\cG_F \co H_1(F;\ZZ) \times H_1(F;\ZZ) \to \ZZ$ called the \textit{Gordon-Litherland pairing}. We show that the signature of this pairing, together with a correction term, gives a signature invariant for $L$, which depends on the choice of spanning surface $F$ only through its $S^*$-equivalence class. The link determinant and nullity invariants are defined similarly, and they are also invariant under $S^*$-equivalence. For classical links, any two spanning surfaces are $S^*$-equivalent, but this is no longer true for links in thickened surfaces. In fact, for a $\ZZ_2$-homologically trivial non-split link in a thickened surface of positive genus, there are two $S^*$-equivalence classes of spanning surfaces. Consequently,  associated to any such link are two signatures, two determinants, and two nullities.  

We highlight some of the main results and applications of the Gordon-Litherland pairing for links in thickened surfaces. The first is Theorem \ref{thm:vg}, which shows that link determinants give simple and easy-to-calculate criteria for a link $L \subset \Si \times I$ to have minimal support genus (see Section \ref{subsec:1-1}). Thus, the link determinants detect the virtual genus.  Another is Theorem \ref{thm:chrom-dual}, which relates the invariants for a given link $L$ with spanning surface $F$ to the invariants for the dual surface $F^*$ under restriction to the kernel of the map $H_1(F^*;\ZZ) \to H_1(\Si \times I;\ZZ)$. We apply this result to checkerboard colorable virtual links to establish a correspondence relating the invariants defined using the Gordon-Litherland pairing with the invariants defined using Goeritz matrices \cite{Im-Lee-Lee-2010}. One important aspect of this correspondence is the principle of chromatic duality, which stipulates that the colors switch under passing from one family of invariants to the other. 

In Corollary \ref{cor-ac1}, we relate the invariants introduced here to the invariants for almost classical links defined using the Seifert pairing \cite{Boden-Chrisman-Gaudreau-2017a}. Thus, our invariants simultaneously generalize those  for checkerboard colorable links  \cite{Im-Lee-Lee-2010}, and those  for almost classical links \cite{Boden-Chrisman-Gaudreau-2017a}. In particular, these two sets of invariants are seen to be equal,  which was not previously known.

The main result in Section \ref{sec-6} is Theorem \ref{thm:isom}, which is an analogue to \cite[Theorem 3]{GL-1978} and  gives a 4-manifold interpretation of the Gordon-Litherland pairing as the intersection form of a double branched cover. Let $F \subset \Si \times I$ be a spanning surface, and let $W$ be a 3-manifold with $\partial W = \Si.$ Theorem \ref{thm:isom} asserts that the pairing $\cG_F$ is equivalent to the intersection pairing on the relative homology of the mirror double cover of $W \times I$ branched along $F$. One curious aspect of this is that the pairing and the intersection form are independent of the choice of the 3-manifold $W$. Associated to a (virtual) spanning surface $F$ is a (virtual) Kirby diagram which gives an explicit description of the mirror double branched cover. Theorem \ref{thm:virtual_GL} then equates the Gordon-Litherland pairing with the (virtual) linking matrix of the associated (virtual) Kirby diagram.
 
There are a number of other results proved here, and we briefly mention a few. For instance, Section \ref{sec:cc} introduces crosscap numbers for virtual knots, and Theorem \ref{thm-min} shows that they can always be realized by minimal genus representatives. Section \ref{sec:vss} introduces virtual spanning surfaces, and Theorem \ref{thm:realization} shows that any allowable integral symmetric matrix occurs as the Gordon-Litherland pairing for some checkerboard colorable virtual knot.  Theorem \ref{thm:cc}  describes the effect of crossing change on the link signature, and Proposition \ref{prop:mirror_image} relates the signature, determinant, and nullity invariants of a link with its horizontal and vertical mirror images.

Here is a short outline of the contents of this paper. In Section \ref{sec-1}, we review the basic notions for links in thickened surfaces and their spanning surfaces. In Section \ref{sec-2}, we introduce the Gordon-Litherland pairing and use it to define   invariants (signature, determinant, and nullity) for links in terms of their spanning surfaces. In Section \ref{sec-3}, we show that the link determinants give sufficient conditions for a link $L \subset \Si \times I$ to have minimal support genus. In Section \ref{sec-4}, we prove a duality result which relates the invariants for a given spanning surface to the restricted invariants for the dual surface. In Section \ref{sec-5}, we  present a simple procedure for computing the signature, determinant, and nullity invariants for a checkerboard colorable link in terms of its Tait graph and associated Goeritz matrix. In Section \ref{sec-6}, we interpret the Gordon-Litherland pairing as the relative intersection form on the 4-manifold given as the mirror double cover of a thickening of a 3-manifold with boundary $\Si$ branched along a copy of the spanning surface pushed into the interior.

\medskip
\noindent\textit{Notation.} 
Decimal numbers such as 2.1 and 3.7 refer to virtual knots in Green's tabulation \cite{green}.

\section{Spanning surfaces for links in thickened surfaces} \label{sec-1}
In this section, we introduce the basic properties for links in thickened surfaces and virtual links, including spanning surfaces, checkerboard colorability, and $S^*$-equivalence. In the last subsection, we introduce crosscap numbers for virtual knots and show that the crosscap numbers are always realized on a minimal genus representative.   

\subsection{Links in thickened surfaces and virtual links} \label{subsec:1-1}
Let  $I = [0,1]$ denote the unit interval and let $\Si$ be a compact, connected, oriented surface.  A \textit{link in the thickened surface} $\Si \times I$ is an embedding $L \co \bigsqcup_{i=1}^\mu S^1 \hookrightarrow \Si \times I$, considered up to isotopy.

A \textit{link diagram} on $\Si$ is an embedded tetravalent graph whose vertices indicate over- and under-crossings in the usual way. Two link diagrams represent isotopic links if and only if they are equivalent by local Reidemeister moves.

Let $p\co \Si \times I \to \Si$ be the projection onto the first factor. For a link $L \subset \Si \times I$, using an isotopy, we can arrange that the image of $L$ under projection $p$ is regular. This means that $p \circ L \co \bigsqcup_{i=1}^\mu S^1 \looparrowright \Si$ is an immersion with finitely many singular points, each of which is a double point. Thus, a regular projection of a link $L \subset \Si \times I$ determines a link diagram of $L$.

Let $\#L$ denote the number of components in the link $L$. A link $L$ with $\#L=1$ is called a \textit{knot}.  An orientation on a link is indicated by placing arrows on the components of its link diagram. Given an oriented link $L$ in $\Si \times I$, we use $-L$ to denote the same link  with the opposite orientation.  

\begin{figure}[h]
\centering
 \includegraphics[scale=1.10]{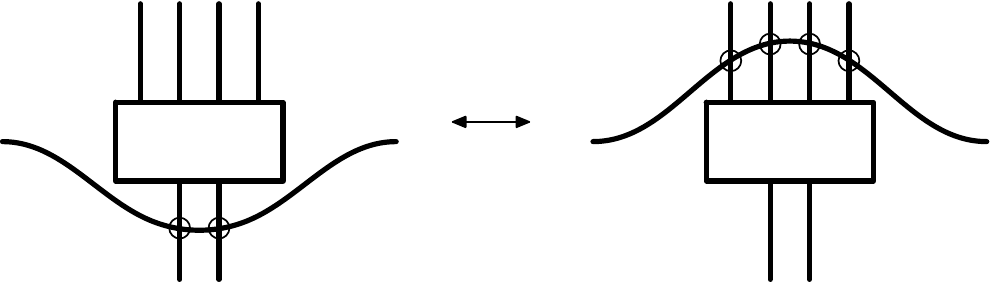} 
 \caption{The detour move.}
\label{detour}
\end{figure} 

A virtual link diagram is an immersion of one or several circles in the plane with only double points, such that each double point is either classical (indicated by over- and under-crossings) or virtual (indicated by a circle). Two virtual link diagrams are said to be equivalent if they can be related by planar isotopies, Reidemeister moves, and the detour move depicted in Figure \ref{detour}. 

Equivalently, a virtual link can be defined as a stable equivalence class of links in thickened surfaces.  Let $L \subset \Si \times I$ be a link in the thickened surface, where $\Si$ is a compact, connected, oriented surface. Stabilization refers to the operation of adding a handle to $\Si$ to obtain a new surface $\Si'$.  Specifically, if $D_0$ and $D_1$ are two disjoint disks in $\Si$ which are disjoint from the image of $L$ under projection $\Si\times I\to \Si$, then $\Si'$ is the surface with $g(\Si') = g(\Si) + 1$ obtained by attaching a 1-handle $h^1 = S^1 \times I$ to $\Si \sm (D_0 \cup D_1)$ so that $\partial h^1 = \partial D_0 \cup \partial D_1$. (Here, $g(\Si)$ denotes the genus of the surface.) This operation is referred to as \textit{stabilization}, and the opposite procedure is called \textit{destablization}. It involves cutting along a vertical annulus in $\Si \times I$ disjoint from the link and attaching two thickened 2-disks. 

Two links $L \subset \Si\times I$ and $L' \subset \Si' \times I$ are said to be \textit{stably equivalent} if one is obtained from the other by a finite sequence of stabilizations, destablizations, and orientation preserving diffeomorphisms of the pairs  $(\Si\times I, \Si \times \{0\})$ and $(\Si'\times I, \Si' \times \{0\})$. In \cite{Carter-Kamada-Saito}, Carter, Kamada, and Saito give a one-to-one correspondence between virtual links and stable equivalence classes of links in thickened surfaces.

A virtual link is called \textit{split}  if it can be represented by a disconnected virtual link diagram in the plane. Likewise, a link $L \subset \Si \times I$ is said to be \textit{split} if it can be represented by a disconnected diagram $D$ on $\Si$. Clearly, a virtual link is split if and only if it can be represented by a split link in a thickened surface.  When that is not the case, we will say that $L$ is \textit{non-split}.

A link diagram $D \subset \Si$ is said to be \textit{cellularly embedded} if $\Si \sm D$ is a union of disks. Of course, given any link $L \subset \Si \times I$, one can successively apply destabilizations until its diagram under projection $p\co \Si \times I \to \Si$ is cellularly embedded. In particular, any virtual link can be represented as a link diagram $D \subset \Si$ which is cellularly embedded.

A link $L \subset \Si \times I$ is said to have \textit{minimal support genus} if it cannot be destabilized. A link of minimal support genus in a closed surface is necessarily cellularly embedded. (The converse to this last statement is false.) In \cite{Kuperberg}, Kuperberg proved that the minimal genus representative of a virtual link is unique up to diffeomorphism. This minimal genus is called the \textit{virtual genus} of the virtual link. By Kuperberg's theorem, if $L \subset \Si \times I$ has minimal support genus, then the associated virtual link has virtual genus equal to $g(\Si).$

  
\subsection{Spanning surfaces and checkerboard colorable links}
Let $L$ be a link in  the thickened surface $\Si \times I$. A \textit{spanning surface} for $L$ is a compact surface $F \subset \Si \times I$ with $\partial F = L.$ Spanning surfaces are not assumed to be connected or oriented, but it will be assumed that they do not contain any closed components. A spanning surface that is oriented and connected will be called a \textit{Seifert surface} for $L$. 

A link diagram $D \subset \Si$ is called \textit{checkerboard colorable} if the components of $\Si \sm D$ can be colored by two colors, say black and white, such that any two components of $\Si \sm D$ that share an edge have different colors. A link  $L \subset \Si \times I$ is said to be \textit{checkerboard colorable} if it admits a checkerboard colorable link diagram in $\Si$.

Suppose the link admits a diagram $D$ which is cellularly embedded and checkerboard colorable. Let $\xi$ denote the  checkerboard coloring. Then the black regions of $\xi$ determine a spanning surface $F_\xi$ for $L$ which we call the \textit{checkerboard surface}. The surface $F_\xi$ is the union of disks and bands, with one disk for each black region and one half-twisted band for each crossing. In constructing $F_\xi \subset \Si \times I$, we place all the disks in $\Si \times \{1/2\}$.

\begin{figure}[ht]
\centering
\includegraphics[scale=1.15]{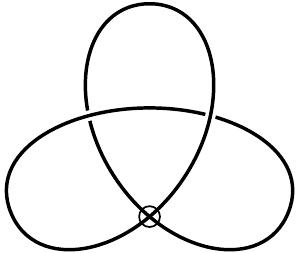} \qquad \qquad \qquad \includegraphics[scale=1.10]{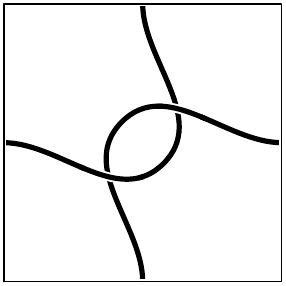}
\caption{The virtual knot 2.1 as a diagram (left) and as a knot on the torus (right).}
\label{2-1}
\end{figure}

\begin{figure}[ht]
\centering
\includegraphics[scale=1.30]{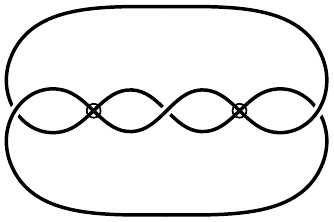} \qquad \qquad  
\includegraphics[scale=1.10]{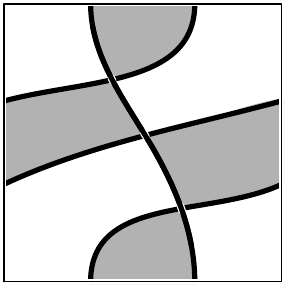} 
\caption{The virtual knot 3.5 as a diagram (left) and as a checkerboard colored knot on the torus (right).}
\label{3-5}
\end{figure}

The next result gives a useful characterization of checkerboard colorability. For a proof, see \cite[Prop. 1.7]{Boden-Karimi-2019}.
 
\begin{proposition} \label{prop:equiv}
Given a link  $L \subset \Si \times I$ in a thickened surface, the following are equivalent:
\begin{enumerate}
\item[(i)] $L$ is checkerboard colorable.
\item[(ii)] $L$ is the boundary of an unoriented spanning surface $F \subset \Si \times I.$
\item[(iii)] $[L]=0$ in the homology group $H_1(\Si \times I; \ZZ_2).$
\end{enumerate}
\end{proposition}

\begin{remark}
In a similar way, one can show that an oriented link $L \subset \Si \times I$ bounds a Seifert surface if and only if it is homologically trivial, i.e., if and only if $[L] = 0$ in $H_1(\Si \times I; \ZZ)$, see \cite{Boden-Gaudreau-Harper-2016}. 
\end{remark}

If $L$ has checkerboard coloring $\xi$, then the \textit{dual coloring} $\xi^*$ is obtained by switching the black and white regions of $\xi$. The dual checkerboard surface $F_{\xi^*}$ therefore coincides with the one given by the white regions of $\xi.$

A virtual link $L$ is said to be \textit{checkerboard colorable} if it can be represented by a checkerboard colorable link in a thickened surface. It is well-known that every classical link is checkerboard colorable, but that is not true in general for links in thickened surfaces or for virtual links. For example, the virtual knot 2.1, shown as a knot in the torus in Figure \ref{2-1}, is not checkerboard colorable, whereas the knot 3.5 in Figure \ref{3-5} is checkerboard colorable.

A virtual link that can be represented as a homologically trivial link in a thickened surface is called \textit{almost classical}. Every almost classical link is checkerboard colorable, but not all checkerboard colorable links are almost classical. For example, the virtual knot 3.5 in Figure \ref{3-5} is checkerboard colorable but not almost classical.


\subsection{Virtual spanning surfaces} \label{sec:vss}
In this subsection, we recall the notion of virtual spanning surfaces. Virtual Seifert surfaces are introduced in \cite{Chrisman-2017}. Here we extend the definition to include nonorientable surfaces. 

\begin{definition} \label{defn:vbs}
A \textit{virtual spanning surface} is a finite union of disjoint disks $D_1,\ldots, D_n$ in $\RR^2$, together with a finite collection of bands $B_1,\ldots, B_m$ in $\RR^2 \sm \Int(D_1 \cup \cdots \cup D_n)$, connecting the disks. Each band can have (classical) twists, and in any region of the plane, at most two bands can intersect. The bands intersect in either  classical or  virtual crossings, as in Figure \ref{vc-band-crossing}. When the surface is oriented, we call it a \textit{virtual Seifert surface}.
\end{definition}

In \cite[Definition 3.1]{Boden-Chrisman-Gaudreau-2017a}, these are called \textit{virtual disk-band surfaces}. In this paper, we reserve that term for when there is just one disk, i.e., for when $n=1$ in the above definition.

\begin{figure}[ht]
\includegraphics[scale=0.80]{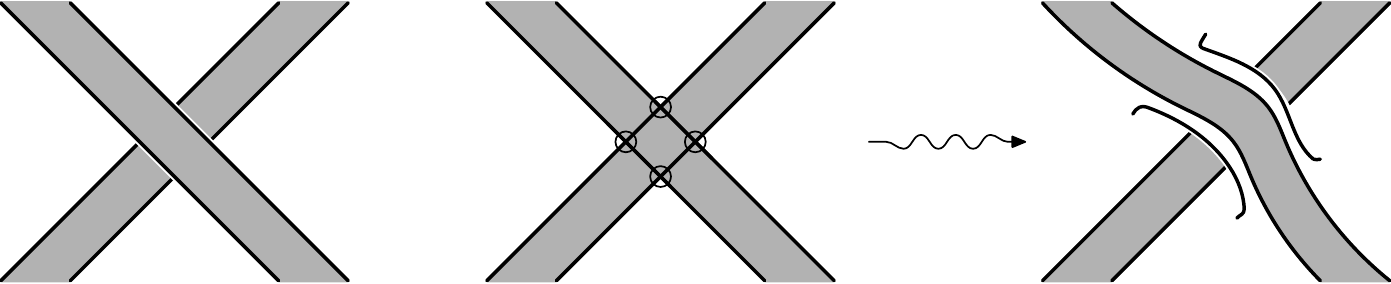}  
\caption{A classical band crossing (left) and virtual band crossing (middle). Attaching a 1-handle allows one of the bands to pass over the other (right).} \label{vc-band-crossing}
\end{figure}

\begin{figure}[ht]
\centering
\includegraphics[scale=.95]{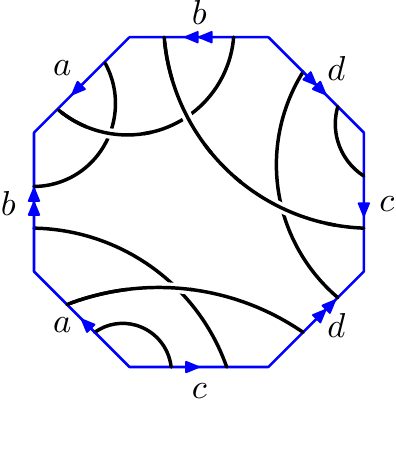}  \hspace{.3cm} 
\includegraphics[scale=.95]{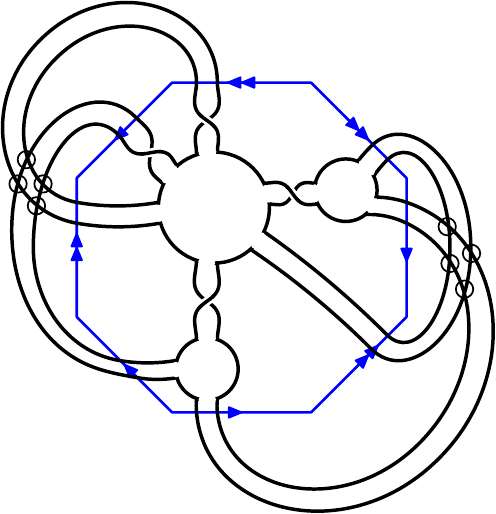}\hspace{.3cm}
\includegraphics[scale=.95]{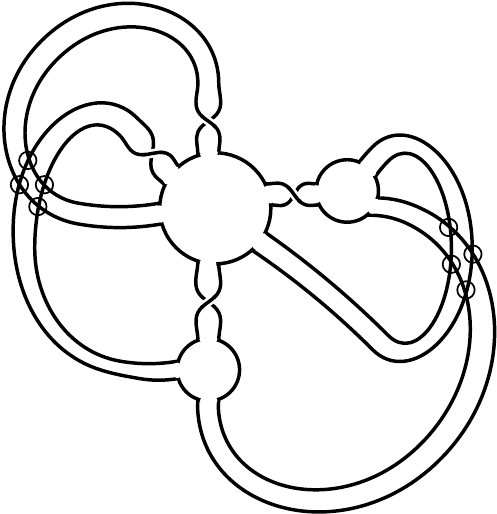} 
\caption{The virtual knot 4.98 and its virtual spanning surface.}
\label{virtual-spanning-surface}
\end{figure}

The ``boundary'' of a virtual spanning surface gives a virtual link diagram. To any virtual spanning surface, we can associate a spanning surface $F$ in a thickened surface $\Si \times I$ in the following natural way. View the virtual spanning surface in $S^2 \times I$, and attach 1-handles to $S^2$ at each virtual band crossing to allow one of the bands to travel along the 1-handle over the other band as in Figure \ref{vc-band-crossing}. The result is a spanning surface in $\Si \times I,$ where $\Si$ has genus equal to the number of virtual band crossings of $F$. The boundary of this spanning surface is a link in $\Si \times I$ representing the virtual link $L$.  Conversely, every spanning surface for a link in $\Si \times I$ can be represented as a virtual spanning surface. This can be proved using a modification of the argument for oriented spanning surfaces, see \cite[Lemma 3.2]{Boden-Chrisman-Gaudreau-2017a} and \cite{Chrisman-2017}. An example of this correspondence is shown in Figure \ref{virtual-spanning-surface}.


\subsection{\boldmath{$S^*$}-equivalence} \label{section:S-star-equivalence}
In this subsection, we recall the notion of $S^*$-equivalence of spanning surfaces, and we establish necessary and sufficient conditions that two spanning surfaces of a given link are $S^*$-equivalent.

Given a spanning surface $F$ and 1-handle $h^1 = I \times D^2 $ in $\Si \times I \sm \partial F$ such that $F \cap h^1 = \partial I \times D^2$, we can construct a new surface by removing the two 2-disks $\partial I \times D^2$ from $F$ and  attaching the annulus $I \times \partial D^2.$ In that case, we say the new surface is obtained from $F$ by attaching a 1-handle.

\begin{definition} \label{defn:S*-equiv}
Two spanning surfaces are \textit{$S^*$-equivalent} if one can be obtained from the other by a finite sequence of the following three moves:
\begin{enumerate}
\item[(a)]Ambient isotopy.
\item[(b)]Attaching a small tube, or its removal.
\item[(c)]Attaching a small half-twisted band, or its removal (see Figure \ref{half-twisted-band}).
\end{enumerate}
\end{definition}

\begin{figure}[ht]
\centering
\includegraphics[scale=1.60]{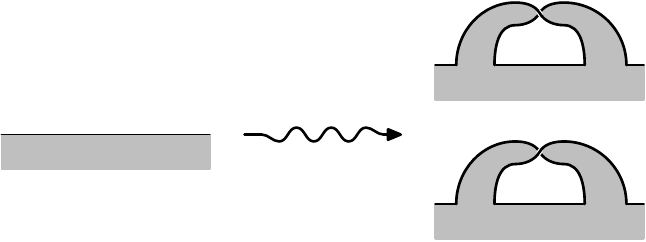} 
\caption{Attaching a small half-twisted band.}
\label{half-twisted-band}
\end{figure}

For classical links, by \cite[Theorem 11]{GL-1978}, any two spanning surfaces are $S^*$-equivalent. For links in thickened surfaces, the situation is more complicated.  Figure \ref{2-1} shows by example that not all knots in thickened surfaces admit spanning surfaces, and Proposition \ref{prop:equiv} implies that a link in $\Si \times I$ admits a spanning surface if and only if it is checkerboard colorable. In that case, the checkerboard surface $F_\xi$ need not be $S^*$-equivalent to the dual checkerboard surface $F_{\xi^*}$. For classical links, an explicit $S^*$-equivalence between the black and white surfaces is given in \cite[Fig.~3]{Yasuhara-2014}, but for links in surfaces of genus $g >0$, the black and white surfaces are not $S^*$-equivalent. This last fact will follow from Equation \eqref{span-and-dual} and Lemma \ref{lemma-S*} below.

Relative homology gives an easy way to distinguish the $S^*$-equivalence classes, as we now explain. Every spanning surface $F$ determines  an element in $H_2(\Si \times I, L;\ZZ_2)$. From the long exact sequence in homology for the pair $(\Si \times I, L)$, we get the exact sequence
$$\xymatrix{
H_2(\Si \times I; \ZZ_2) \ar[r]^{j_*\ \ } &H_2(\Si \times I, L; \ZZ_2) \ar[r]^{\ \quad \de_*} & H_1(L; \ZZ_2) \ar[r]^{i_*\ \ \ }& H_1(\Si \times I;\ZZ_2).}$$ 
Notice that the first map $j_*$ is injective, since $H_2(L;\ZZ_2)=0$, and also that the fundamental class $[L]$ lies in the image of the second map $\de_*$ since $L$ being checkerboard colorable implies that $[L]$ is trivial in $H_1(\Si \times I;\ZZ_2)$. 

It follows that $H_2(\Si \times I, L; \ZZ_2)$ has rank $r \geq 2$. This uses the facts that $H_2(\Si \times I; \ZZ_2) \cong \ZZ_2$, with generator the fundamental class $[\Si],$  and that $H_1(L; \ZZ_2) \cong (\ZZ_2)^\mu$, with generators $[K_1], \ldots, [K_\mu]$ given by the components of  $L = K_1 \cup \cdots \cup K_\mu$. Notice further that if $F_\xi$ is a checkerboard surface and $F_{\xi^*}$ is the dual surface, then 
\begin{equation} \label{span-and-dual}
[F_\xi] + [F_{\xi^*}] = j_*([\Si]).
\end{equation}
In particular, it follows that $[F_\xi]\neq [F_{\xi^*}]$ in $H_2(\Si \times I, L; \ZZ_2).$ 

\begin{lemma} \label{lemma-S*}
Suppose $L \subset \Si \times I$ is a checkerboard colorable link in a thickened surface of genus $g(\Si)>0.$ If $F_1$ and $F_2$ are $S^*$-equivalent spanning surfaces for $L$, then $[F_1] = [F_2]$ as elements in $H_2(\Si \times I, L; \ZZ_2).$
\end{lemma}

\begin{proof}
This follows by showing that the relative homology class of a spanning surface is invariant under the three moves (a), (b), (c) of Definition \ref{defn:S*-equiv} that generate $S^*$-equivalence. For the first move, this holds because homology is an isotopy invariant. For the second move, if $F_2$ is obtained from $F_1$ by the addition of a small tube, then $F_1$ and $F_2$ cobound the 3-manifold $F_1 \cup -F_2 \cup \text{(1-handle)}$ in $\Si \times I$. This implies that $F_1$ and $F_2$ are homologous. 

For the third move, since we have already seen that every spanning surface is $S^*$-equivalent to a checkerboard surface by isotopies and 1-handle moves, we can assume that $F_1$ is a checkerboard surface. In that case, notice that the dual surface $F_2^*$ is obtained from $F_1^*$ by an isotopy. Thus, by the previous argument, we see  that $[F_1^*] = [F_2^*]$, and so Equation \eqref{span-and-dual} implies that $[F_1]=[F_2]$ whenever $F_2$ is obtained from $F_1$ by the addition of a half-twisted band. This completes the proof.
\end{proof}
 
The next result is a generalization of \cite[Theorem 11]{GL-1978} for checkerboard colorable links in thickened surfaces.

\begin{proposition} \label{prop-S*}
Let $L \subset \Si \times I$ be a non-split link admitting a checkerboard colored diagram $D$ on $\Si$. Then any spanning surface for $L$ is $S^*$-equivalent to the checkerboard surface $F_\xi$ or its dual $F_{\xi^*}$.
\end{proposition}

\begin{proof}
If  $D$ and $D'$ are two checkerboard colorable diagrams for $L$ with colorings $\xi$ and $\xi'$, respectively, then  $(D', \xi')$ is equivalent through checkerboard colored diagrams to either $(D, \xi)$ or its dual $(D,\xi^*)$. (For a proof, see \cite[Theorem 3.3]{Im-Lee-Lee-2010}.) An argument similar to the one in \cite{Yasuhara-2014} shows that the checkerboard surfaces of equivalent diagrams are $S^*$-equivalent, thus it follows that $F_{\xi'}$ is $S^*$-equivalent to either $F_{\xi}$ or $F_{\xi^*}$. 

Thus, the proposition follows once we show that any spanning surface $F$ is $S^*$-equivalent to a checkerboard surface. In the proof of Proposition \ref{prop:equiv}, we showed that, given any spanning surface $F$ for $L$, after performing isotopy and attaching 1-handles, the new surface is a checkerboard surface for a diagram of $L$. The new surface is clearly $S^*$-equivalent to $F$, in fact, in the equivalence we only need moves of type (a) and (b). 
\end{proof}

Notice that by combining Proposition \ref{prop-S*} and Lemma \ref{lemma-S*}, it follows that if $L \subset \Si \times I$ is a non-split checkerboard colorable link in a thickened surface of genus $g(\Si)>0$, then two spanning surfaces $F_1$ and $F_2$ for $L$ are $S^*$-equivalent if and only if $[F_1] = [F_2]$ as elements of $H_2(\Si \times I, L;\ZZ_2)$.
 

\subsection{Crosscap numbers} \label{sec:cc}
We will now define crosscap numbers for knots in thickened surfaces and for virtual knots. We begin by recalling several definitions.

Every closed connected nonorientable surface $F$ is homeomorphic to a connected sum of copies of the real projective plane $\RP^2$. The \textit{Euler genus} of a connected nonorientable surface $F$ is the positive integer $k$ such that $F=\#_{i=1}^k \RP^2$.  For such a surface $F$, we write $\ga(F)=k$ for its Euler genus. In case $F$ has nonempty boundary, we define $\ga(F)$ to be the Euler genus of the closed surface obtained by capping each component of $\partial F$ with a disk.

Every knot $K$ in $S^3$ bounds a nonorientable surface, and the crosscap number of $K$ is the minimum Euler genus over all nonorientable spanning surfaces for $K$. By convention, the crosscap number of the unknot is set to be zero.

\begin{definition} 
Let $K\subset \Si \times I$ be a knot with checkerboard coloring $\xi$, and let $F_\xi$ be its black checkerboard surface. The \textit{$\xi$-crosscap number} of $K$, denoted $C_\xi(K)$, is defined to be the minimum Euler genus over all nonorientable spanning surfaces $F$ for $K$ which are $S^*$-equivalent to $F_\xi$. 
 \end{definition}

The checkerboard surface $F_{\xi^*}$ for the dual coloring $\xi^*$ is the white checkerboard surface for $\xi.$ Since $F_\xi$ and $F_{\xi^*}$ are not $S^*$-equivalent (unless $g(\Si)=0$), the two crosscap numbers $C_{\xi}(K)$ and $C_{\xi^*}(K)$ need not be equal. Thus, knots in thickened surfaces of genus $g>0$ typically have two crosscap numbers. Examples of this phenomenon will appear in a forthcoming paper by the second author.

The next result is a generalization to nonorientable spanning surfaces of \cite[Theorem 6.4]{Boden-Gaudreau-Harper-2016}. It shows that the $\xi$-crosscap number of a knot is monotone non-increasing  under destabilization.  
We say that an annulus  $A$ in $\Si \times I$ is \textit{vertical} if $A$ is isotopic rel boundary to $\ga \times I \subset \Si \times I$ for some embedded circle $\ga \subset \Si.$

\begin{theorem}   \label{thm-min}
Let $K \subset \Si \times I$ be a checkerboard colorable knot and $F$ a nonorientable connected spanning surface for $K$. Fix the coloring $\xi$ of $K$ so that $F$ is $S^*$-equivalent to $F_\xi$, 
the black checkerboard surface. 

Suppose $A$ is a vertical annulus in $\Si \times I$  disjoint from $K.$ Let $\Si'$ be the surface obtained from $\Si$ by destabilization along $A$, and let $K'$ be the resulting knot in $\Si' \times I$. (If the destabilization along $A$ separates $\Si$, we choose $\Si'$ to be the component containing $K'$.) Then $\xi$ extends to a coloring $\xi'$ of $K'$ in $\Si' \times I$, and there exists a nonorientable connected spanning surface $F'$ for $K'$ in $\Si' \times I$, which is $S^*$-equivalent to the black checkerboard surface $F_{\xi'}$ with $\ga(F') \leq \ga(F)$.
\end{theorem}

\begin{proof} Let $A$ be a vertical annulus for $\Si \times I$, and consider the intersection $F \cap A$. If $F \cap A$ is empty, then we can take $F' = F.$  Otherwise, we can arrange that $A$ and $F$ intersect transversely, so that $A \cap F$  consists of a union of circles in the interiors of $A$ and $F$. Cutting $\Si \times I$ along $A$ then cuts $F$ along these circles. Denote the resulting surface $F'$, and let $\Si'$ be the destabilized surface, so $\Si' \times I$ is obtained from the closure of $\Si \times I \sm N(A)$ by capping the holes with thickened disks. 

In $\Si' \times I,$ there are two disjoint copies of $A$. If $A$ separates $\Si \times I,$ then it is customary to discard the component of $\Si \times I$ that does not contain $K'$. Each component $C$ of $A \cap F$ appears as a circle in both copies of $A$, and both circles bound disks in $\Si' \times I$. These disks can be chosen to be pairwise disjoint. Attaching them to the components of $\partial F'$ along the cut circles, we obtain a spanning surface $F''$ for $K'$. Notice that $F''$ is nonorientable, since it contains one-sided circles, and that $b_1(F'') \leq \ga(F).$ (Here, $b_1(F'')$ refers to the rank of $H_1(F'';\ZZ_2),$ i.e., the first Betti number of $F''$ over $\ZZ_2$.)

If the surface $F''$ is disconnected, one can construct a connected spanning surface by simply discarding the closed components of $F''.$ The resulting surface, denoted $F'''$, is the component of $F''$ with non-empty boundary. Clearly it is a connected spanning surface for $K'$ with $b_1(F''') \leq b_1(F)$. If $F'''$ is nonorientable, then $\ga(F''') \leq \ga(F)$ and we are done. Otherwise, if $F'''$ is orientable, then one of the discarded components of $F''$ must be nonorientable with positive Euler genus. Therefore $b_1(F''') \leq  \ga(F'')-1.$ Attaching a half-twisted band to $F'''$, we obtain a nonorientable surface with Euler genus $\ga=b_1(F''')+1 \leq  \ga(F'')$.
\end{proof}

Theorem \ref{thm-min} applies to show that crosscap numbers are well-defined for checkerboard colorable virtual knots and can always be computed using a minimal genus representative.

The next result is obtained by combining Proposition \ref{prop:equiv} and Theorem \ref{thm-min}.
\begin{proposition}
If $L$ is a checkerboard colorable virtual link, then any  link diagram $D  \subset \Si$ of minimal support genus that represents $L$ is checkerboard colorable.
\end{proposition}

\begin{proof}
If $L$ is a checkerboard colorable virtual link, then it can be represented by a link in a thickened surface $\Si \times I$ which bounds a spanning surface. Theorem \ref{thm-min} then implies that there exists a representative of minimal support genus which bounds a spanning surface. By Proposition \ref{prop:equiv}, we see that $L$ can be represented by a minimal genus diagram $D \subset \Si$ with $[D]=0$ in $H_1(\Si;\ZZ_2).$ Kuperberg's theorem \cite{Kuperberg} then implies the same is true for any minimal genus diagram representing $L$. In particular, it follows that any minimal genus diagram for $L$ is checkerboard colorable.
\end{proof}

The next result summarizes our discussion. The proof is immediate and left to the reader. Note further that, this result implies that, for classical knots,  their crosscap numbers in the virtual category are equal to their crosscap numbers as classical knots.
\begin{corollary}
Given a checkerboard colorable virtual knot, its crosscap numbers are given by $C_\xi(K)$ and $ C_{\xi^*}(K)$, where $K \subset \Si \times I$ is a minimal genus representative, $\xi$ is a coloring for $K$, and $\xi^*$ is its dual coloring.
\end{corollary}

\section{The Gordon-Litherland pairing} \label{sec-2}
In this section, we introduce the Gordon-Litherland pairing for $\ZZ_2$-homologically trivial links in thickened surfaces. We use the pairing to define signature, determinant, and nullity invariants for links in thickened surface,  and we show that the invariants depend only on the $S^*$-equivalence class of the spanning surface. 


\subsection{An asymmetric linking} \label{sec:rel-link}
 We begin by describing an asymmetric linking for disjoint knots in $\Si \times I$ called the relative linking. Given disjoint oriented knots $J,K \subset \Si \times I$, the (relative) linking number $\lk_\Si(J,K)$ is defined as the algebraic intersection number $J\cdot B$, where $B$ is a 2-chain in $\Si \times I$ with $\partial B = K +c$ for some 1-cycle $c$ in $\Si \times \{1\}$. An easy exercise shows this is independent of the choice of relative 2-chain $B$. 

The relative linking numbers are not symmetric; instead they satisfy   
\begin{equation} \label{eq:linking}
\lk_\Si(J,K) - \lk_\Si(K,J) = p_* [J]\cdot p_*[K],
\end{equation}
where $\cdot$ is the algebraic intersection number in $\Si$ (see \cite[\S10.2]{Cimasoni-Turaev}). 

We adopt the convention that linking in $\Si \times I$ is computed relative to the top $\Si \times \{1\}.$ However, one can also consider relative linking relative to the bottom  $\Si \times \{0\}.$ It is sometimes necessary to refer to both forms of relative linking, and in that case we use $\overline{\lk}_\Si(J,K)$ for linking relative to the top and $\underline{\lk}_\Si(J,K)$ for linking relative to the bottom. 

If $J$ and $K$ are disjoint oriented knots, then $\overline{\lk}_\Si(J,K)$ is computed by counting, with sign, the number of times that $J$ crosses above $K$ in $\Si \times I$, where ``above'' is with respect to the positive $I$-direction in $\Si \times I$. On the other hand, $\underline{\lk}_\Si(J,K)$ is defined as the algebraic intersection number $J\cdot B$, where $B$ is a 2-chain in $\Si \times I$ with $\partial B = K +c$ for some 1-cycle $c$ in $\Si \times \{0\}$. It is computed by counting, with sign, the number of times that $J$ crosses below $K$ in $\Si \times I$. 

Notice that the top and bottom relative linkings satisfy $\underline{\lk}_\Si(J,K) = \overline{\lk}_\Si(K,J).$ Alternatively, under the orientation reversing diffeomorphism $\Si \times I \to \Si \times I$ given by sending $(x,t) \mapsto (x,1-t)$, we see that $\overline{\lk}_\Si(J,K)$ transforms into $\underline{\lk}_\Si(J,K)$ and vice versa.


\subsection{A symmetric bilinear pairing}
We now describe the Gordon-Litherland pairing, extending the methods of \cite{GL-1978} to the present setting. Let $L$ be a link in a thickened surface $\Si\times I$ with unoriented spanning surface $F$. A closed tubular neighborhood $N$ of $F$ in $\Si \times I$ is a  $[-1,1]$-bundle over $F$, and we set $\wt{F}$ to be the associated $\{\pm 1\}$-bundle. So $\wt{F} \to F$ is the orientable double cover when $F$ is not orientable, and it is the trivial double cover otherwise.

Define a pairing $\cG_F \co H_1(F;\ZZ) \times H_1(F;\ZZ) \lto \ZZ$ by setting 
\begin{equation} \label{defn:GL}
\cG_F(\al,\be) = \lk_\Si(\tau \al, \be) - p_*(\al) \cdot p_*(\be) , 
\end{equation}
where $p_* \co H_1(\Si \times I;\ZZ) \to H_1(\Si;\ZZ)$ is the induced homomorphism of the projection $p\co\ \Si \times I \to \Si$, $\lk_\Si$ is the relative linking number, $\tau \co H_1(F;\ZZ) \to H_1(\wt{F};\ZZ)$ is the transfer map for the double cover $\wt{F} \to F,$ and $p_*(\al) \cdot p_*(\be) $ is the algebraic intersection of the homology classes $p_*(\al), p_*(\be) \in H_1(\Si;\ZZ)$.

\begin{lemma} \label{lem:sym}
$\cG_F$ is symmetric.
\end{lemma}

\begin{proof}
The proof is similar to the one given in \cite[Proposition 9]{GL-1978}, and we include it for completeness. Orient $\wt{F}$ so that the positive normal vector points out of $N,$ and let $i_\pm \co \wt{F} \to \Si \times I \sm \wt{F}$ be the positive and negative push-offs, respectively. For $\al,\be \in H_1(F;\ZZ),$
\begin{equation*}
\begin{split}
\cG_F(\al,\be) &= \lk_\Si(\tau \al,\be) -p_*(\al) \cdot p_*(\be) \\
&= \lk_\Si \left((i_+)_* (\tau \al),\be\right) -p_*(\al) \cdot p_*(\be) \\
&= \tfrac{1}{2} \lk_\Si\left((i_+)_* (\tau \al),\tau \be\right) -p_*(\al) \cdot p_*(\be).
\end{split}
\end{equation*}
 
Therefore, applying the above formula twice and  applying also Equation \eqref{eq:linking} to the curves $(i_-)_* (\tau \al)$ and $\tau \be$ in the second line below, we find that
\begin{equation*}
\begin{split}
\cG_F(\al,\be) - \cG_F(\be,\al) 
&=  \tfrac{1}{2} \lk_\Si\left((i_+)_* (\tau \al),\tau \be\right) - p_*(\al) \cdot p_*(\be) \\
& \quad -\left(\tfrac{1}{2} \lk_\Si\left((i_+)_* (\tau \be),\tau \al\right)- p_*(\be) \cdot p_*(\al)\right) \\
&=  \tfrac{1}{2} \lk_\Si \left((i_+)_* (\tau \al),\tau \be\right) -\tfrac{1}{2} \lk_\Si\left(\tau \be,(i_-)_*(\tau \al)\right) - 2 \, p_*(\al) \cdot p_*(\be) \\
&= \tfrac{1}{2} \lk_\Si \left((i_+)_* (\tau \al),\tau \be\right) -\tfrac{1}{2} \lk_\Si\left((i_-)_* (\tau \al),\tau \be\right) \\
&= \tfrac{1}{2} \lk_\Si\left((i_+)_* (\tau \al)-(i_-)_*(\tau \al),\tau \be\right)\\
&= \tfrac{1}{2}\tau \al \cdot \tau \be \\
&=  0. 
\end{split}
\end{equation*}

Here we use the fact that $\tfrac{1}{2} p_*(\tau \al)\cdot p_*(\tau \be) = 2\, p_*(\al) \cdot p_*(\be) .$ Note that the term $\tau \al\cdot \tau \be$ in the last step refers to the  algebraic intersection of curves on $\wt{F},$ and since each point of $\al \cap \be$ gives rise to two points in $\tau \al \cap \tau \be$ of opposite signs, it follows that $\tau \al \cdot \tau \be =0$. This completes the proof of the lemma.
\end{proof}

Suppose that $L \subset \Si \times I$ is a $\ZZ$-homologically trivial link and $F \subset \Si \times I$ is a Seifert surface. Then the Seifert matrices $V^\pm$ are the $n \times n$ matrices with $i,j$ entry equal to $\lk_\Si(a_i^\pm,a_j)$, where $\{a_1,\ldots, a_n\}$ is a set of simple closed curves on $F$ giving an ordered basis for $H_1(F;\ZZ)$, and $a_i^+$ and $a_i^-$ denote the positive and negative push-offs of $a_i$ with respect to an oriented bi-collaring of $F$ in $\Si \times I$. (Linking numbers here refer to the relative linking introduced in Section \ref{sec:rel-link}.)

By \cite[Lemma 2.1]{Boden-Chrisman-Gaudreau-2017a}, the signature and nullity of the symmetrized Seifert matrices satisfy
$$\sig(V^{+}+\,(V^{+})^\tr) =\sig(V^{-}+\,(V^{-})^\tr) \,  \text{ and } \, 
\nullity(V^{+}+\,(V^{+})^\tr) = \nullity(V^{-}+\,(V^{-})^\tr).$$
We can  therefore define the signature and nullity by setting
\begin{equation} \label{eqn:sign}
\si(L,F)= \sig(V^{\pm}+(V^{\pm})^\tr) \, \text{ and } \,
n(L,F)=\nullity(V^{\pm}+(V^{\pm})^\tr).
\end{equation}

The next result shows that the Gordon-Litherland pairing specializes to the symmetrized Seifert pairing when the spanning surface is oriented.

\begin{theorem} \label{thm-ac-GL}
Suppose $L$ is the boundary of a Seifert surface $F$. 
Then $\cG_{F}$ is represented by the symmetric matrix $V^- + \, (V^-)^\tr$, where $V^-$ is the Seifert matrix associated to $F$ obtained by taking the negative push-offs.
\end{theorem}

\begin{proof}
Since $F$ is orientable, $\tau \al = (i_-)_*(\al) + (i_+)_*(\al).$
Thus 
\begin{equation*}
\begin{split}
\cG_{F}(\al,\be) &= \lk_\Si(\tau \al, \be) - p_*(\al)\cdot p_*(\be) \\
&= \lk_\Si \left((i_-)_*(\al), \be\right) +\lk_\Si \left((i_+)_*(\al), \be\right)- p_*(\al)\cdot p_*(\be) \\
&= \lk_\Si \left((i_-)_*(\al), \be\right) + \lk_\Si \left(\al, (i_-)_* \be\right) - p_*(\al)\cdot p_*(\be) \\
&= \lk_\Si \left((i_-)_*(\al), \be\right) + \lk_\Si \left((i_-)_* \be, \al \right). 
\end{split}
\end{equation*}

In the above, Equation \eqref{eq:linking} is applied in the third line to the curves $\al$ and $(i_-)_* (\be)$.

If $\al_1,\ldots, \al_{2g}$ is a basis for $H_1(F;\ZZ)$, then the symmetrized Seifert matrix $V^- + (V^-)^\tr$ has $(i,j)$ entry given by 
$\lk_\Si(\al_i^-,\al_j) + \lk_\Si(\al_j^-,\al_i)$. Thus 
$$\quad \qquad \cG_{F}(\al_i,\al_j) =\lk_\Si \left((i_-)_*(\al_i), \al_j\right) +\lk_\Si\left((i_-)_* (\al_j), \al_i\right)
= V^-_{ij} + (V^-)^\tr_{ij}. \qquad \quad \qedhere$$
\end{proof}

Notice further that when $F$ is oriented, we have $e(F) = -\lk_\Si(L,L')=0$. Theorem \ref{thm-ac-GL} implies that $\sig(\cG_F) = \sig(V^- + (V^-)^\tr) =  \si(L,F).$ If $L$ is a link admitting a Seifert surfaces, then every spanning surface is $S^*$-equivalent to an oriented spanning surface $F$. Thus the $S^*$-invariant quantity identified in  Lemma \ref{lemma:S*-inv} is equal to the  signature  $\si(L,F)$ defined in Equation \eqref{eqn:sign} for some Seifert surface $F$.


\subsection{Link invariants} \label{subsec:sign-GL}
In this subsection, we describe the link signature, determinant, and nullity invariants derived from the Gordon-Litherland pairing. These invariants are shown to depend only on the $S^*$-equivalence
class of the spanning surface.

Suppose $L \subset \Si \times I$ is a link with $\mu$ components $L = K_1\cup \cdots \cup K_\mu$. Let $F \subset \Si \times I$ be a spanning surface and $L'$ be the push-off of $L$ which misses $F$. Then we can write $L' = K'_1\cup \cdots \cup K'_\mu$. Define  
\begin{equation} \label{eqn:Euler}
e(F) = - \sum_{i=1}^\mu \lk_\Si(K_i,K_i'),
\end{equation} 
where each pair $K_i, K_i'$ is oriented compatibly and $\lk_\Si(\,\cdot\,,\,\cdot\,)$ refers to the relative linking of Section \ref{sec:rel-link}.

\begin{lemma} \label{lemma:S*-inv}
If $F_1$ and $F_2$ are $S^*$-equivalent spanning surfaces in $\Si \times I$, then $$\sig(\cG_{F_1}) + \tfrac{1}{2} e(F_1) = \sig(\cG_{F_2}) + \tfrac{1}{2} e(F_2).$$
\end{lemma}
\begin{proof} The proof is identical to the proof of \cite[Proposition 10]{GL-1978}.
\end{proof}

Lemma \ref{lemma:S*-inv} implies that the quantity $\sig(\cG_{F}) + \frac{1}{2} e(F)$ depends only on the $S^*$-equivalence class of $F$, and Proposition \ref{prop-S*} implies that a checkerboard colorable knot admits at most two $S^*$-equivalence classes of spanning surfaces. Note that  since neither $\sig(\cG_{F})$ nor $e(F)$ depend on having chosen an orientation of $L$, the quantity $\sig(\cG_{F}) + \tfrac{1}{2} e(F)$ is an invariant of the unoriented link $L \subset \Si \times I.$ It is the analogue, for links in thickened surfaces, of the Murasugi invariant of classical links \cite{Murasugi-1970}.

Now suppose $L$ is given an orientation, and suppose further that $F$ is a spanning surface for $L$ and $L'$ is a push-off of $L$ missing $F$. Define  
$$e(F,L) = -\lk_\Si(L,L') = - \sum_{i,j=1}^\mu \lk_\Si(K_i,K_j').$$
Then one can easily verify that $e(F,L) = e(F) - \la(L)$, where $\la(L) = \sum_{i\neq j} \lk_\Si(K_i,K_j)$ denotes the total linking number of $L$. 

We can  therefore define the signature  by setting
\begin{equation} \label{defn:signature}
\si(L,F)= \sig(\cG_F) + \tfrac{1}{2} e(F,L)
\end{equation}
Lemma \ref{lemma:S*-inv} implies that $\si(L,F)$ is a well-defined link invariant that depends only on the $S^*$-equivalence class of $F$.

We note that, just in the case of classical links, the signature invariant $\si(L,F)$ is unchanged if the orientation on each component of $L$ is reversed. Writing 
$$\si(L,F) = \sig({\cG}_{F}) + \tfrac{1}{2}e(F,L) =\sig({\cG}_{F}) + \tfrac{1}{2}\left(e(F) - \la(L)\right),$$
it follows that $\si(L,F)$ depends on the orientation of $L$ only through the total linking number $\la(L).$ For instance, if $L=K_1 \cup K_2 \cup  \cdots \cup K_\mu$ and $L' = -K_1 \cup K_2 \cup \cdots \cup K_\mu$ is the result of reversing the orientation of the first component, then 
$$\si(L',F)-\si(L,F)= \tfrac{1}{2}\left(\la(L)-\la(L')\right) = \sum_{i=2}^\mu \lk_\Si(K_1,K_i)+\lk_\Si(K_i,K_1).$$

\begin{figure}[ht]
\hfill
\includegraphics[width=6cm]{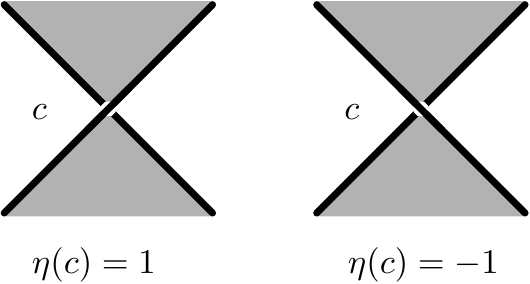}\hfill \qquad\quad 
\includegraphics[width=6cm]{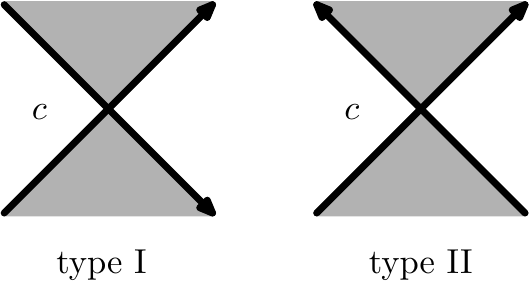} \hfill
\hfill
\caption{The incidence number and type of a crossing.}
\label{crossing-eta-type}
\end{figure}

Suppose $D$ is a link diagram for a link $L$ and $\xi$ is a checkerboard coloring with associated checkerboard surface $F_\xi$. The incidence number $\eta(c)$ of a crossing $c \in C_D$ is defined with respect to the coloring as in Figure \ref{crossing-eta-type}, and a crossing $c \in C_D$ is said to be type I or type II according to Figure \ref{crossing-eta-type}. 

The next result relates $e(F_\xi,L)$ with the quantity 
\begin{equation} \label{eqn:correction}
\mu_\xi(D) = \sum_{\text{$c$ type II}}  \eta(c),
\end{equation}
the sum of the incidence numbers over type II crossings of $D$.

\begin{lemma} \label{lemma:correction}
$e(F_\xi,L) = -2 \mu_\xi(D).$ 
\end{lemma}
\begin{proof}
Let $D'$ be the push-off of $D$ that misses $F_\xi$. Then $D'$ can  be taken  to lie on $\Si \times \{1/2\}$ except in a small neighborhood of each of the crossings. Thus, the quantity $e(F_\xi,L) = -\lk_\Si(D,D')$ can be calculated as a sum of contributions, one for each crossing  of $D$. A routine exercise shows that a crossing $c \in C_D$ contributes to $\lk_\Si(D,D')$ according to its type; if it is type I then it contributes $0$, and if it is type II then it contributes $2 \eta(c)$. Comparing with Equation \eqref{eqn:correction}, it follows that
$$\qquad \qquad \qquad e(F_\xi,L) = -\lk_\Si(D,D') = -2 \sum_{\text{$c$ type II}} \eta(c) = -2 \mu_\xi(D). \qquad \qquad \qquad \qedhere$$
\end{proof}

One can also use the Gordon-Litherland pairing to define link invariants of $L$ in terms of its determinant and nullity as follows. Suppose $F$ is a connected spanning surface for a link $L \subset \Si \times I$. Given a basis for  $H_1(F;\ZZ),$ we can write out the matrix representative for the Gordon-Litherland pairing $\cG_F$. Under a change of basis this matrix will change by unimodular congruence. Therefore, the determinant and nullity of $\cG_F$ give well-defined invariants, denoted by
\begin{equation} \label{defn:det}
\det(L,F)=|\det(\cG_F)| \, \text{ and } \, 
n(L,F)=\nullity(\cG_F).
\end{equation}
Just as with the link signature in Equation \eqref{eqn:correction}, we will see that these two quantities depend only on the $S^*$-equivalence class of the connected surface $F$. (In case of a disconnected spanning surface, one can define the nullity by taking $n(L,F)=\nullity(\cG_F)+b_0(F)-1,$ where $b_0(F)=\dim H_0(F;\ZZ)$.)

\begin{theorem} Let $F_1,F_2$ be two connected spanning surfaces for a link $L \subset \Si \times I$. 
If $F_1$ and $F_2$ are $S^*$-equivalent,  then we have
\[
\left|\det(\cG_{F_1})\right|=\left|\det(\cG_{F_2})\right|
\, \text{ and } \, 
\nullity(\cG_{F_1})=\nullity(\cG_{F_2}).
\]
\end{theorem}
\begin{proof} 
For $i=1,2$, choose a basis for $H_1(F_i;\ZZ)$ and let $G_{i}$ be the matrix representative for $\cG_{F_i}$.
If $F_1$ and $F_2$ are ambient isotopic, then the result holds since $G_1$ and $G_2$ are unimodular congruent in that case. Suppose then that $F_2$ is obtained from $F_1$ by adding a thin tube. Then we can find bases so that $G_{2}$ is equal to
\[
\begin{bmatrix} G_{1} & \ga & \textbf{0} \\ \ga^{\intercal} & 0 & 1 \\ \textbf{0} & 1 & 0\end{bmatrix} \equiv \begin{bmatrix} G_{1} & \textbf{0} & \textbf{0} \\ \textbf{0} & 0 & 1 \\ \textbf{0} & 1 & 0\end{bmatrix},
\]
where $\equiv$ denotes unimodular congruence. Again we conclude that $\left|\det(G_{1})\right|=\left|\det(G_{2})\right|$. If $F_2$ is obtained from $F_1$ by adding a half-twisted band, as in Figure \ref{half-twisted-band}, then the rank of $H_1(F_2;\ZZ)$ is one greater than that of $H_1(F_1;\ZZ)$.  Let $\de$ be the generator given by the core of the half-twisted band that is added. Then $\lk_{\Si}(\tau\de,\de)=\pm 1$, depending on the direction of the half twist. We can find bases so that  $G_{2}$ is equal to
\[
\begin{bmatrix} G_{1} & \textbf{0} \\ \textbf{0} & \pm 1 \end{bmatrix}.
\]  
Again, we have that $\left|\det(G_{1})\right|=\left|\det(G_{2})\right|$. Since $S^*$-equivalence is generated by these three operations, the theorem is proved.
\end{proof}

\begin{example} \label{ex:3-7}
\begin{figure}[ht]
\centering
\includegraphics[scale=0.90]{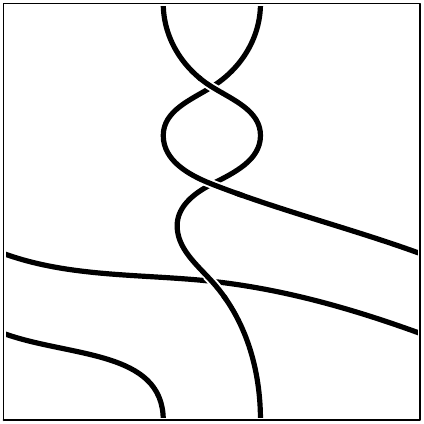}
\qquad
 \includegraphics[scale=0.90]{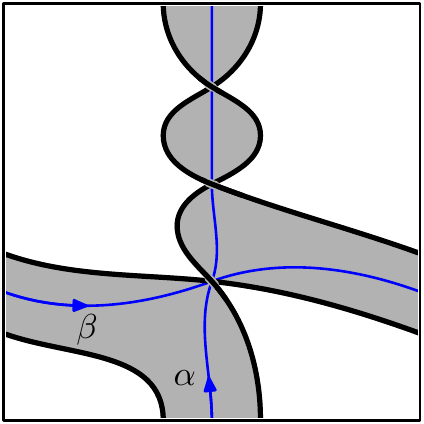}
\qquad
 \includegraphics[scale=0.90]{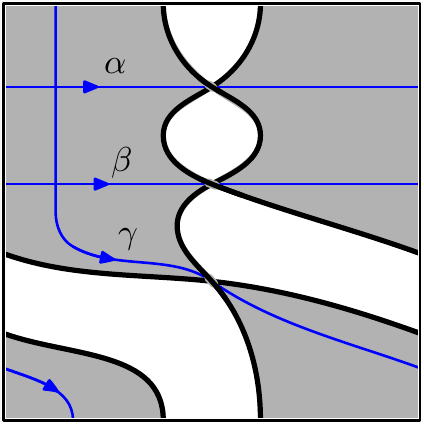}
\caption{An alternating knot with dual checkerboard colorings and generators for homology.} \label{3-7-a}
\end{figure}

Consider the knot in the torus in Figure \ref{3-7-a}  with its two colorings. Let $F$ be the black surface in the first coloring (middle) and $F^*$ the black surface in the dual coloring (right).

Using the basis $\al,\be$ for $H_1(F;\ZZ)$ in Figure \ref{3-7-a} (middle), we compute that  $\cG_{F}$ has matrix ${\footnotesize \begin{pmatrix} -3 & -1 \\ -1 & -1 \end{pmatrix}}$. If $K'$ is a parallel to $K$ missing $F$, then Equation \eqref{eqn:Euler} gives that $e(F) = -\lk_\Si(K,K') =4$. Therefore, $\sig(\cG_F)=-2$, and it follows that 
$$\si(K,F)= \sig(\cG_F)+e(F)/2= -2 + 2 =0 \, \text{ and } \, \det(K,F)=2.$$

Using the basis $\al,\be,\ga$ for $H_1(F^*;\ZZ)$ in Figure \ref{3-7-a} (right), we compute that $\cG_{F^*}$ has matrix ${\footnotesize \begin{pmatrix} 1 & 0 & 0 \\ 0&1&0 \\ 0&0&1 \end{pmatrix}}.$ We take a moment to explain this step.

Firstly, since each of $\al,\be,\ga$ pass through only one crossing with $\eta=1,$ it follows that $\cG_{F^*}(\al,\al)=1=\cG_{F^*}(\be,\be) = \cG_{F^*}(\ga,\ga)$. Furthermore, since $\al$ and $\be$ are disjoint curves, we have $\cG_{F^*}(\al,\be) = 0$. In fact, both $\cG_{F^*}(\al,\ga)$ and $\cG_{F^*}(\be,\ga)$ vanish as well, even though the curves are not disjoint. For example, 
$$\cG_{F^*}(\al,\ga) = \lk_\Si(\tau \al, \ga) - p_*(\al)\cdot p_*(\ga) = -1 -(-1) =0,$$ 
with a similar argument for $\cG_{F^*}(\be,\ga)$. 

Clearly $\sig(\cG_{F^*})=3$. If $K''$ is a parallel that misses $F^*$, then Equation \eqref{eqn:Euler} implies that $e(F^*) =-\lk_\Si(K,K'')=-2.$ Thus,  
$$\si(K,F^*)= \sig(\cG_{F^*})+e(F^*)/2= 3 -1 =2 \, \text{ and } \, \det(K,F^*)=1.$$
 \hfill $\Diamond$
\end{example}


\section{Detecting the virtual genus} \label{sec-3}
In this section, we apply the link determinants to detect the virtual genus of non-split virtual links. This is achieved by establishing a criterion for any checkerboard colorable link $L \subset \Si \times I$ in a thickened surface of minimal genus.

Suppose $L \subset \Si \times I$ is a non-split link in a thickened surface whose associated link diagram $D$ is cellularly embedded and checkerboard colored. Let $F_\xi$ be the checkerboard surface and $F_{\xi^*}$ the dual surface.  The next result shows that if $\Si$ is not a minimal genus surface for $L$, then either $\det(L,F_\xi)=0$ or $\det(L,F_{\xi^*})=0$.

\begin{theorem} \label{thm:vg}
Let $L \subset \Si \times I$ be a non-split checkerboard colorable link. If $L$ is not a minimal genus representative, then one of the determinants of $L$ is zero.
\end{theorem}

\begin{proof}
Let $D \subset \Si$ be a diagram for $L$, which is assumed to be cellularly embedded and checkerboard colorable. In particular, the diagram $D$ has minimal support genus $g(\Si).$ However, since $L$ is not a minimal genus representative, it is isotopic to a link $L'$ whose diagram $D'$ does admit a destabilization. 

Notice that $D'$ is also checkerboard colorable, but it is not cellularly embedded. Therefore, we have a non-contractible simple closed curve $\ga$ in $\Si$ disjoint from $D'$. Since $L$ is non-split, $D'$ is necessarily connected. Since $\ga \cap D' = \varnothing,$ it follows that $\ga$ is contained entirely in either a black region or a white region in a coloring of $D'$.  (Without loss of generality, we can assume the coloring $\xi$ has been chosen so that $\ga$ lies in a black region.)

Let $F_\xi$ be the associated checkerboard surface for $L'$. We claim that there is a simple closed curve $\al$ lying entirely in a black region of $\xi$ such that its homology class $[\al]$ is nontrivial as an element in $H_1(F_\xi;\ZZ).$ Indeed, if $\ga$ is a non-separating curve, then we can take $\al =\ga.$ Otherwise, if $\ga$ is a separating curve, then since $D'$ is connected, it lies in one of the connected components of $\Si \sm \ga$. Both connected components have positive genus, and one of them is contained entirely in a black region of the coloring $\xi$. Therefore, we can take $\al$ to be a simple closed curve in the component disjoint from $D'$, and we can further choose $\al$ so that $[\al]\neq 0$. 
 
Since $\al$ is a simple closed curve, its homology class $[\al]$ is primitive as an element in $H_1(F_\xi;\ZZ)$. Therefore, we can find a basis $B$ for $H_1(F_\xi;\ZZ)$ with $\al \in B.$ Further, since $\al$ lies entirely within the black region, we have $\cG_{F_\xi}(\al, \al)=0$. For any other basis element $\be \in B$, $\al$ and $\be$ will intersect transversely in a finite number of points within the black region. Further, one can check that each point of intersection contributes 0 to $\cG_F(\al,\be)$. (This step follows by a similar argument as used in Example \ref{ex:3-7} when we showed $\cG_{F^*}(\al,\ga)=0=\cG_{F^*}(\be,\ga)$ for $F^*, \al,\be,\ga$ in Figure \ref{3-7-a} (right).) Therefore, $\cG_F(\al,\be)=0$ for all $\be \in B.$

It follows that the Gordon-Litherland pairing $\cG_{F_\xi}$ is singular, therefore $n(L',F_\xi) \neq 0.$ Since $L'$ and $L$ are isotopic links, it follows that one of the nullities of $L$ is necessarily nonzero. In particular, one of the determinants of $L$ is equal to zero.
\end{proof}

\begin{figure}[ht]
\centering
\includegraphics[scale=.85]{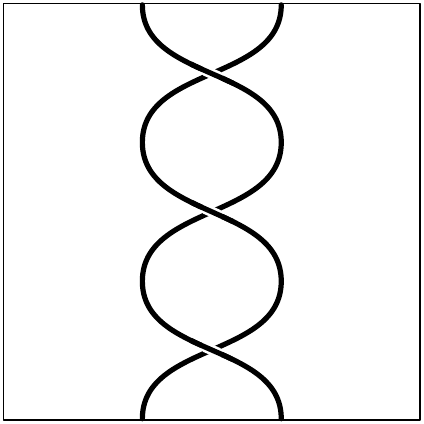}
 \includegraphics[scale=.85]{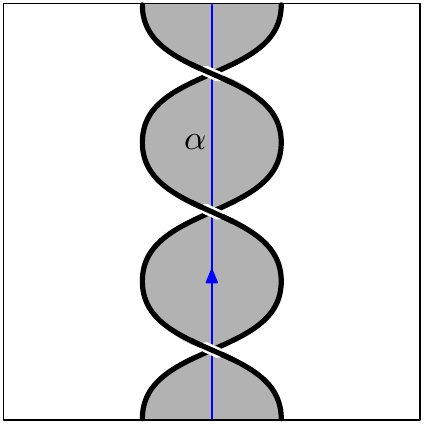}
 \includegraphics[scale=.85]{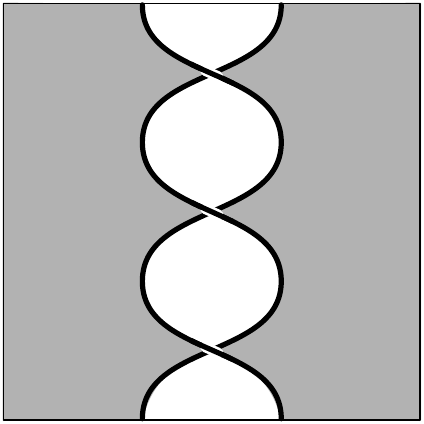}
 \includegraphics[scale=.85]{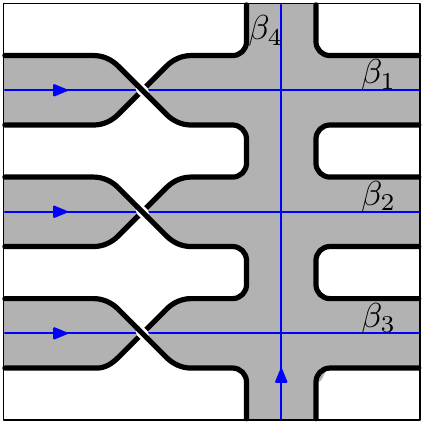}
\caption{A diagram of the trefoil on the torus with dual checkerboard colorings and generators for homology.} \label{non-minimal-trefoil}
\end{figure}

\begin{example}\label{ex:non-minimal-trefoil}
Consider the knot diagram in Figure \ref{non-minimal-trefoil}. Let $F$ be the checkerboard surface in Figure \ref{non-minimal-trefoil} (second from the left) and $F^*$ be the dual checkerboard surface in Figure \ref{non-minimal-trefoil} (third from the left). In terms of the basis $\{\al\}$ pictured, the Gordon-Litherland pairing $\cG_F$ is represented by the matrix $[-3],$ which has signature $-1$. Taking $K'$ a parallel that misses $F$ and applying Equation \eqref{eqn:Euler}, we see that $e(F) = -\lk_\Si(K,K') = 6.$ Therefore, $\si(K,F) =\sig(\cG_F) + e(F)/2 = 2$, and $\det(K,F)= |\det(\cG_F)| = 3$.

The dual surface $F^*$ admits an isotopy to the rightmost picture in Figure \ref{non-minimal-trefoil}. Using the basis $\{\be_1, \dots, \be_4\}$ pictured there,  
the Gordon-Litherland pairing $\cG_{F^*}$ is represented by the matrix
$$\begin{bmatrix} 
1 & 0 & 0 & 0  \\
0 & 1 & 0 & 0  \\
0 & 0 & 1 & 0 \\
0 & 0 & 0 & 0 
\end{bmatrix}.$$
So $\sig(\cG_{F^*}) =3$ and $\det(\cG_{F^*})=0.$ Further, if $K''$ is a parallel missing $F^*$, then by Equation \eqref{eqn:Euler}, it follows that $e(F^*)= -\lk_\Si(K,K'') =0.$ Therefore $\si(K,F^*) =3$ and $\det(K,F^*)=0$. 
\hfill $\Diamond$ \end{example}

As a virtual knot, the diagram in Figure \ref{non-minimal-trefoil} is a non-minimal genus representative of the (classical) trefoil. This example shows that the signatures, determinants, and nullities are not generally invariant under stabilization and destabilization. 

Despite this shortcoming, the invariants derived from the Gordon-Litherland pairing can nevertheless be used to give well-defined invariants of checkerboard colorable virtual links. This relies on combining Kuperberg's theorem with the observation that signatures, determinants, and nullities are invariant under homeomorphisms of the pair $(\Si \times I, \Si \times \{0\})$, with the proviso that one must compute them using a minimal genus representative. Note that, by the discussion in Section \ref{sec:cc}, if $L\subset \Si \times I$ is a minimal genus representative  of a checkerboard colorable virtual link, then $L$ itself admits a checkerboard coloring.

\begin{figure}[ht]
\centering
\includegraphics[scale=1.2]{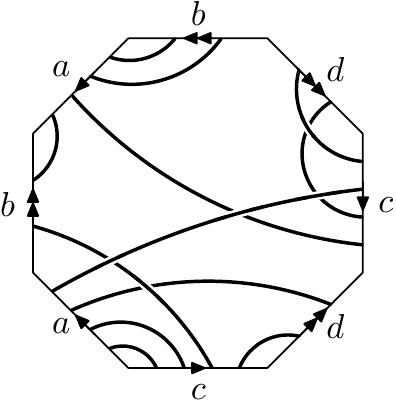}
\quad
 \includegraphics[scale=1.2]{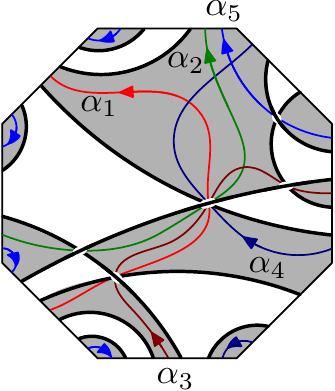}
\quad
 \includegraphics[scale=1.2]{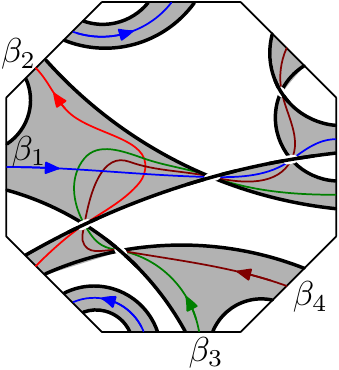}
\caption{
A diagram of the virtual knot 5.2024 on a genus two surface with dual checkerboard colorings and generators for homology.} \label{fig:5-2024}
\end{figure}

The next example shows that the converse to Theorem \ref{thm:vg} is not generally true. In fact, a minimal genus link in a thickened surface may have one or even both determinants equal to zero.

\begin{example} \label{ex:5-2024}
The virtual knot 5.2024 can be represented as a knot $K$ in a thickened surface $\Si \times I$ of genus 2 (see Figure \ref{fig:5-2024}). Since this representative has minimal crossing number and is cellularly embedded, it is a minimal genus representative \cite{Manturov-2013}. However, as we shall see, this knot does not satisfy the hypothesis of Theorem \ref{thm:vg}.

Let $F$ be the checkerboard surface in Figure \ref{fig:5-2024} (middle) and $F^*$ be the dual checkerboard surface in Figure \ref{fig:5-2024} (right). Then in terms of the basis $\{\al_1, \dots, \al_5\}$ pictured, the Gordon-Litherland pairing $\cG_F$ is represented by the matrix
$$\begin{bmatrix} 
0 & 1 & 0 & 1 & 0 \\
1 & 0 & 1 & 1 & 0 \\
0 & 1 & -1 & 1 & 0 \\
1 & 1 & 1 & 1 & 0 \\
0 & 0 & 0 & 0 & 1 \\
\end{bmatrix}.$$
One can easily compute the signature and determinant of this matrix, giving $\sig(\cG_F) =-1$ and $\det(\cG_F)=-1.$ If $K'$ is a parallel that misses $F$, then by Equation \eqref{eqn:Euler}, we see that $e(F) = -\lk_\Si(K,K')=2.$ Therefore, $\si(K,F) =\sig(\cG_F) + e(F)/2 = 0$, and $\det(K,F)= |\det(\cG_F)| = 1$.

For the dual surface $F^*$, using the basis $\{\be_1, \dots, \be_4\}$ pictured, the  Gordon-Litherland pairing $\cG_{F^*}$ is represented by the matrix
$$\begin{bmatrix} 
0 & 0 & -1 & 0  \\
0 & 1 & 1 & 1  \\
-1 & 1 & 1 & -1 \\
0 & 1 & -1 & 1 
\end{bmatrix}.$$
Computing its signature and determinant shows that $\sig(\cG_{F^*}) =1$ and $\det(\cG_{F^*})=0.$ If $K''$ is a parallel that misses $F^*$, then we again using Equation \eqref{eqn:Euler}, we find that $e(F^*)= -\lk_\Si(K,K'')=0.$ Therefore, $\si(K,F^*) =1$, and $\det(K,F^*)=0$.
\hfill $\Diamond$ \end{example}

The  signature, determinant, and nullity can also be computed directly from virtual spanning surfaces. This is particularly convenient when working with checkerboard colorable virtual links.

\begin{figure}[ht]
\centering
\includegraphics[scale=1]{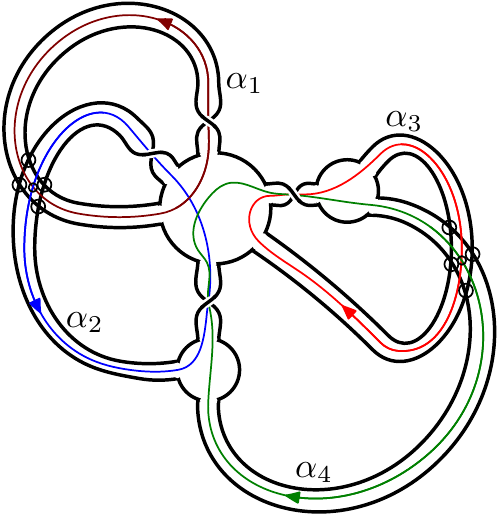}
\caption{A virtual spanning surface for the virtual knot 4.98 and generators for homology.} \label{fig:4-98-gen}
\end{figure}

\begin{example} Let $F$ be the virtual spanning surface for the virtual knot 4.98 pictured in Figure \ref{fig:4-98-gen} and consider the basis $\{\al_1,\dots,\al_4\}$ for $H_1(F;\ZZ)$ pictured there. Then $\cG_F$ is represented by the matrix
$$\begin{bmatrix} 
-1 & 0 & 0 & 0  \\
0 & 0 & 0 & 1  \\
0 & 0 & 1 & 1 \\
0 & 1 & 1 & 2 
\end{bmatrix}.$$

For any virtual knot $K$ with virtual spanning surface $F$, it is not difficult to verify that the Euler number is given by $e(F) = -\vlk(K,K')$, where $K'$ is the parallel to the virtual knot $K$ that misses $F$ and $\vlk(\cdot, \cdot)$ refers to the virtual linking number (see Section \ref{sec:vln}). This follows from Equation \eqref{eqn:Euler} and the observation that, under the correspondence between links in thickened surfaces and virtual links, we have $\lk_\Si(\cdot,\cdot) = \vlk(\cdot, \cdot)$ (again, see Section \ref{sec:vln}).

For this example, one can directly compute that $e(F) = -\vlk(K,K')=0$. An elementary calculation shows that $\sig(\cG_F)=0$. Thus  $\si(K,F)=0$, $n(K,F)=0$, and $\det(K,F)=1$.
\hfill $\Diamond$ \end{example}

Let $M =(m_{ij})$ be a symmetric $n \times n$  matrix over the integers. We say $M$ is \emph{allowable} if either $n$ is even or $m_{ii}$ is odd for some $i$. The next result shows that any allowable symmetric, integral matrix occurs as a representative of the Gordon-Litherland pairing for some virtual spanning surface. It is the analogue, for checkerboard colorable virtual knots, of Theorem 3.7 \cite{Boden-Chrisman-Gaudreau-2017a}, where a realization result for Seifert pairs is proved for almost classical knots.

\begin{theorem}\label{thm:realization}
Any integral $n \times n$ symmetric matrix that is allowable, represents the Gordon-Litherland pairing for some checkerboard colorable virtual knot.
\end{theorem}

We shall prove this by constructing a virtual spanning surface with boundary a virtual knot and whose Gordon-Litherland pairing is the given matrix. The surface constructed will have first homology of rank $n$, and it will be orientable if and only if all the diagonal entries of the matrix are even. Since the first homology of an orientable surface with connected boundary must have even rank, this explains why the matrix is required to be allowable.

\begin{proof}
The proof is similar to \cite[Theorem 3.7]{Boden-Chrisman-Gaudreau-2017a}, and so we provide a sketch. 

Let $M =(m_{ij})$ be a symmetric $n \times n$  matrix over the integers, and assume $M$ is allowable. Note that if the theorem is true for $M$, then it is also true for any matrix obtained from $M$ under unimodular congruence. Since $M$ is allowable, then either $n$ is even or some diagonal entry of $M$ is odd. In the latter case, we can arrange, by a unimodular congruence, that the last diagonal entry $m_{nn}$ is odd.

We will construct a virtual spanning surface whose associated Gordon-Litherland pairing is represented by $M$. Start with a 2-disk in $\RR^2$ sitting below the $x$ axis with the line segment $\{(t,0) \mid 0\leq t \leq 2n\}$ on its boundary. If all the diagonal entries of $M$ are even, then $n$ is necessarily even. In this case, for $1<i<n/2,$ we attach the bands in pairs, so that the feet are alternating, i.e., so that the feet of the $(2i-1)$-st and $(2i)$-th bands are centered at the pairs of points $(4i-3,0), (4i-1,0)$ and $(4i-2,0), (4i,0)$, respectively. The band crossing between them should be drawn as a virtual band crossing.

If some diagonal entry of $M$ is odd, then we attach the bands in pairs as follows. For $1<i<\lfloor n/2 \rfloor$, the $i$-th pair consists of the $(2i-1)$-st and $(2i)$-th bands. If one or both of $m_{2i-1, 2i-1}, m_{2i, 2i}$ is even, then we attach the bands as above with their feet alternating and so that the band crossing between them is virtual.

If instead both $m_{2i-1, 2i-1}$ and $m_{2i, 2i}$ are odd, then we attach the  bands with their feet nested, i.e., so that the feet of the $(2i-1)$-st and $(2i)$-th bands are centered at the pairs of points $(4i-3,0), (4i-2,0)$ and $(4i-1,0), (4i,0)$, respectively.   In this case, there is no band crossing between them.

In case $n$ is odd, there will be one additional band, and recall that we have arranged that $m_{nn}$ is odd. This last band should be attached with its feet centered at  $(2n-1,0)$ and $(2n,0)$. An easy proof by induction shows that the resulting virtual spanning surface will bound a virtual knot, and we leave the details to the indefatigable reader.

With the bands in place, next we arrange for them to have the correct self-linking. For $1 \leq i \leq n$, we insert $|m_{ii}|$ half twists into the $i$-th band, where the twists are right-handed if $m_{ii}>0$ and left-handed if $m_{ii}<0$. 

The last step is to arrange for the correct linking between the $i$-th and $j$-th bands. Fix orientations on each of the bands so its core runs from left to right. For $i<j$, insert a sequence of band crossings paired with virtual band crossings between the $i$-th and $j$-th bands, so that $i$-th band crosses over the $j$-th band $|m_{ij}|$ times. (See \cite[Figure 11]{Boden-Chrisman-Gaudreau-2017a} for an illustration.) Here, with respect to the orientations on the bands, the band crossings are positive if $m_{ij}>0$ and negative if $m_{ij}<0$. It may be necessary for the $i$-th band to cross some of the other bands that are in the way, and that can be achieved using virtual band crossings. 

The resulting virtual spanning surface has $H_1(F;\ZZ) = \ZZ^n$, with the cores of the bands as a generating set. Its Gordon-Litherland pairing is easily seen to be represented by the matrix $M$.
\end{proof}


\section{Duality} \label{sec-4}
In this section, we prove a duality result relating the invariants of one $S^*$-equivalence class of spanning surfaces to those obtained under restriction to the other $S^*$-equivalence class.  This result has practical value in that it allows one to compute both sets of invariants from one spanning surface. 

Let $g =g(\Si)$ be the  genus of the surface $\Si$, and let $L \subset \Si \times I$ be a link with diagram $D$ which is cellularly embedded. Consequently, it follows that the inclusion map $i \co D \to \Si$ induces a surjection $i_*\co H_1(D;\ZZ) \to H_1(\Si;\ZZ).$

Given any spanning surface $F$ for $L$, we can construct a new surface $F'=F\#_\tau \Si$ by connecting it to a parallel copy of the Carter surface near $\Si \times \{0\}$ to $F$ by a small thin tube $\tau$. Clearly $[F] + [F\#_\tau \Si] = j_*([\Si])$ in $H_2(\Si \times I, L; \ZZ_2).$ Thus, $F$ and $F'$ are not $S^*$-equivalent (unless $L$ is a link in $S^2 \times I$). Since $F$ and $F'$ have same local behaviour near $L$, Lemma \ref{lemma:correction} implies that $e(F',L) = e(F,L)$. 

\begin{theorem} \label{thm:chrom-dual}
Let $F \subset \Si \times I$ be a connected spanning surface such that the map $H_1(F;\ZZ) \to H_1(\Si \times I;\ZZ)$ is surjective. Set  $\cK_F=\,\text{Ker}\,(H_1(F;\ZZ) \to H_1(\Si \times I;\ZZ))$ and let $\cG_{F}|_{\cK_F}$ denote the restriction of $\cG_{F}$ to $\cK_F.$Then the link signature, determinant, and nullity for $F'$ are equal to those of $\cG_{F}|_{\cK_F}$, i.e., 
$$\sig(\cG_{F'}) = \sig(\cG_{F}|_{\cK_F}),\;\; |\det(\cG_{F'})| = |\det(\cG_{F}|_{\cK_F})|, \text{ and } \nullity(\cG_{F'}) = \nullity(\cG_{F}|_{\cK_F}).$$
 \end{theorem}
\begin{remark} \label{rem:convention}
The signature and determinant of the empty matrix are 0 and 1 by convention.
\end{remark}

\begin{proof}
Since $H_1(F;\ZZ) \to H_1(\Si \times I;\ZZ)$ is surjective, there is a basis for $H_1(F;\ZZ)$ consisting of curves $\{\al_1,\ldots, \al_{n}; \ga_1,\ldots, \ga_{2g}\}$ in $F$  such that  $\al_1,\ldots, \al_{n} \subset \cK_F$ and $\ga_1,\ldots, \ga_{2g}$ map to a standard symplectic basis for $H_1(\Si;\ZZ)$.

We can extend this to a basis for $H_1(F';\ZZ)$  of the form

$$\be'=\{ \al_1,\ldots, \al_n; \ga_1,\ldots, \ga_{2g}; \ga'_1,\ldots, \ga'_{2g} \},$$ 
where $\ga'_1, \ldots, \ga'_{2g}$ are the images of a standard symplectic basis for $H_1(\Si;\ZZ)$ in the parallel copy of the Carter surface that is attached to $F$ in forming $F'$. Then the Gordon-Litherland matrix with respect to the basis $\be'$ has block decomposition as the symmetric matrix:
\begin{equation} \label{eqn:big-matrix}
\begin{bmatrix} A&*& 0 \\ *&B & J_g \\ 0&J_g^\tr & 0 \end{bmatrix},
\end{equation}
where $A$ is the $n \times n$ matrix for the restriction of $\cG_{F}$ to $\cK_F$, $B$ is the $2g \times 2g$ matrix obtained by restricting the Gordon-Litherland form $\cG_{F}$ to $\{\ga_1,\ldots, \ga_{2g}\}$ and $$J_g=\begin{bmatrix} 0& I_g \\ -I_g & 0 \end{bmatrix} $$ is the standard $2g \times 2g$ symplectic matrix representing the intersection form on $\Si$. (Here $I_g$ denotes the $g \times g$ identity matrix.) A straightforward exercise in linear algebra then shows that the matrix in Equation \eqref{eqn:big-matrix} is unimodular congruent to one of the form:
$$\begin{bmatrix} A&0& 0 \\ 0&B & J_g \\ 0&J_g^\tr & 0 \end{bmatrix}.$$
This can be achieved using only row and column operations from the last two blocks of rows and columns. Consequently, the submatrix $A$ is unchanged throughout these operations. The signature of the above matrix is easily seen to be equal to that of $A$, since signature is additive over block orthogonal decompositions, and since $$\sig \left( \begin{bmatrix} B& J_g \\ J_g^\tr & 0 \end{bmatrix}\right) = 0.$$
It follows that $\sig(\cG_{F'}) = \sig(\cG_{F}|_{\cK_F}).$

Notice that the determinant of a symmetric matrix is invariant under unimodular congruence up to sign. Therefore, arguing as above, we see that
\begin{equation*} 
\det \begin{bmatrix} A&*& 0 \\ *&B & J_g \\ 0&J_g^\tr & 0 \end{bmatrix} = \pm \det \begin{bmatrix} A&0& 0 \\ 0&0 & J_g \\ 0&J_g^\tr & 0 \end{bmatrix}=\pm\det A.
\end{equation*}
This implies that $|\det(\cG_{F'})| = |\det(\cG_{F}|_{\cK_F})|,$ and equality of $\nullity(\cG_{F'})$ and $\nullity(\cG_{F}|_{\cK_F})$ follows similarly. This completes the proof.
\end{proof}

\begin{remark} One could further arrange that $B$ is diagonal, using unimodular congruence over $\ZZ$. In addition, one could eliminate the diagonal entries by working over $\QQ$. This step may require dividing by 2 in the row and/or column operations. 
\end{remark}
 
Theorem \ref{thm:chrom-dual} allows one to compute both sets of invariants from one spanning surface.  This is illustrated in the next example, which concerns the alternating knot in the torus in Figure \ref{3-7-a}, whose invariants were computed in Example \ref{ex:3-7}.
\begin{example} \label{ex:conv}
Applying Theorem \ref{thm:chrom-dual} to the first surface $F$ in Figure \ref{3-7-a} (middle), we compute the signature, determinant, and nullity invariants for the second surface $F^*$ in Figure \ref{3-7-a} (right). Set $\cK_F=\text{Ker} \,(H_1(F;\ZZ) \to H_1(\Si;\ZZ))$, and note that $\cK_F =0$. Therefore, $\sig(\cG_F|_{\cK_F}) =0.$ Since $e(F)=4,$ it follows that $\si(K,F^*) = 2,  \; \det(K,F^*) =1,$ and $n(K,F^*)=0$ (see Remark \ref{rem:convention}). These values agree with those obtained in Example \ref{ex:3-7}, but with considerable simplification.
\hfill $\Diamond$ \end{example}

\begin{example} \label{ex:conv2}
This example is like Example \ref{ex:conv} but with the roles of the two surfaces reversed. Namely, we apply Theorem \ref{thm:chrom-dual} to the second surface $F^*$ in Figure \ref{3-7-a} (right) and use it to compute the signature, determinant, and nullity invariants for the first surface $F$ in Figure \ref{3-7-a} (middle). Set $\cK_{F^*}=\text{Ker} \,(H_1(F^*;\ZZ) \to H_1(\Si;\ZZ))$. Then $\cK_{F^*}$ is generated by $\al-\be$. Since $\cG_{F^*}(\al,\be)=0$, it follows that 
$$\cG_{F^*}(\al-\be, \al-\be) =\cG_{F^*}(\al, \al) +\cG_{F^*}(\be, \be)=1+1=2.$$

Therefore, $\cG_{F^*}|_{\cK_{F^*}}$ is represented by the matrix  $\begin{bmatrix} 2 \end{bmatrix}$. Since $e(F^*)=-2,$ it follows that $\si(K,F) = \sig(\cG_{F^*}|_{\cK_{F^*}}) +e(F^*)/2 = 1 -1 =0,  \; \det(K,F) =2,$ and $n(K,F)=0$. These values agree with those obtained in Example \ref{ex:3-7}.
\hfill $\Diamond$ \end{example}

\begin{example} \label{ex:conv-5-2024}
In a similar way, we can apply Theorem \ref{thm:chrom-dual} to the second surface $F^*$ in Figure \ref{fig:5-2024} (right) to simplify the computations of the signature, determinant, and nullity invariants for the first surface $F$ in Figure \ref{fig:5-2024} (middle). 

Set $\cK_{F^*}=\text{Ker} \,(H_1(F^*;\ZZ) \to H_1(\Si;\ZZ))$, and note that once again we have $\cK_{F^*} =0$. Therefore, $\sig(\cG_{F^*}|_{\cK_{F^*}}) =0.$ Since $e(F^*)=0,$ it follows immediately that $\si(K,F) = 0,  \; \det(K,F) =1,$ and $n(K,F)=0$ (cf., Example \ref{ex:5-2024}).
\hfill $\Diamond$ \end{example}

In Section \ref{sec-5}, we will apply Theorem \ref{thm:chrom-dual} to relate the link invariants coming from the Gordon-Litherland pairing to the combinatorial invariants defined by Im, Lee, and Lee  \cite{Im-Lee-Lee-2010}.

Next, we apply Theorem \ref{thm:chrom-dual} to give a bound on the  difference of the two signatures of a checkerboard colorable link $L \subset \Si \times I$ in terms of the nullities and the genus of $\Si$. 

To that end, we recall a well-known and useful method for computing signatures of symmetric matrices in terms of chains of principal minors. The following is a restatement of \cite[Proposition~13.32]{Burde-Zieschang-Heusener}.

Let $Q$ be a symmetric real matrix of rank $r$. Then there exists a chain $M_0, M_1,\dots, M_r$ of principal minors of $Q$ with $M_0=1$ such that, for $0\leq i < r,$
\begin{itemize}
\item[(i)]  $M_i$ is a principal minor of $M_{i+1},$ and 
\item[(ii)] no two consecutive determinants $M_i$ and $M_{i+1}$ vanish. 
\end{itemize}
Then the signature of $Q$ is given by
\begin{equation}\label{sig-formula}
\si(Q) =\sum_{i=0}^{r-1} \text{sign}(M_{i}M_{i+1}).
\end{equation}

Now let $n,n'$ be positive integers with $n>n'$, and let $M$ be a symmetric $n\times n$ matrix defined over $\QQ$. Let $M'$ be the  submatrix of $M$ of size $n'\times n'$ such that $M'_{ij}=M_{ij}$ for $ 1\leq i,j\leq n'$. In other words, there is a block decomposition of matrices
$$M=\left[\begin{matrix} M'&M''\\ (M'')^{\tr}&M'''\end{matrix}\right].$$ 

Let $\si_M,\si_{M'}$ denote the signatures of $M,M'$, respectively, and $k,k'$ their nullities. Therefore, $\text{rank}(M)=n-k$ and $\text{rank}(M')=n'-k'$. Clearly, $ \text{rank}(M) \geq \text{rank}(M')$, thus $n-k \geq n'-k'$. In particular, $(n-n')-(k-k')\geq 0.$
   
\begin{lemma}\label{sig-nul}
We have 
\begin{equation}\label{original}
|\si_M-\si_{M'}|\leq (n-n')-(k-k').
\end{equation}
\end{lemma}

\begin{proof}
Choose a chain of principal minors $M_1,M_2,\ldots,M_{n-k}$ of $M$, where each $M_i$ is an $i\times i$ submatrix of $M$. Further, we can arrange that $M_1,M_2,\ldots,M_{n'-k'}$ is a chain of principal minors of $M'$, and that $\det (M_{n'-k'})\neq 0$, and $\det (M_{n-k})\neq 0$. Lastly, we assume that no two consecutive minors have zero determinant.

By Equation \eqref{sig-formula}, 
\begin{equation}\label{(1)}
\si(M_{n'-k'})-[(n-k)-(n'-k')]\leq \si(M_{n-k})\leq \si(M_{n'-k'})+[(n-k)-(n'-k')].
\end{equation}
Notice that $\si_{M'}=\si(M_{n'-k'})$, and $\si_{M}=\si(M_{n-k})$. Therefore, Equation \eqref{(1)} shows that
\begin{equation*} 
-[(n-n')-(k-k')] \leq \si_{M}-\si_{M'}\leq (n-n')-(k-k'),
\end{equation*}
and Equation \eqref{original} follows.
\end{proof}

\begin{corollary}
Let $L \subset \Si \times I$ be a checkerboard colorable link with spanning surface $F$. If $F'$ is another spanning surface which is not $S^*$-equivalent to $F$, then 
$$|\si(L,F) - \si(L,F')| + |n(L,F) - n(L,F')| \leq 2g(\Si).$$
In particular, if $\det(L,F)\neq 0$ and $\det(L,F')\neq 0,$ then $|\si(L,F) - \si(L,F')| \leq 2g(\Si).$
\end{corollary}

\begin{proof}
Since $L$ has at most two $S^*$-equivalence classes of spanning surfaces, it follows that $F'$ is $S^*$-equivalent to $F\#_\tau \Si$. Therefore, 
$$|\si(L,F) - \si(L,F')| = |\si(L,F) - \si(L,F \#_\tau \Si)| = |\sig(\cG_F) - \sig(\cG_{F \#_\tau \Si})|.$$

Lemma \ref{sig-nul} applies to show that 
$$|\sig(\cG_{F \#_\tau \Si})-\sig(\cG_F)| \leq 2g - n(L,F \#_\tau \Si)+n(L,F).$$
Further, since $F$ must be $S^*$-equivalent to $F' \#_\tau \Si$, the same argument with the surfaces reversed shows that
$$|\si(L,F') - \si(L,F)| = |\sig(\cG_{F' \#_\tau \Si})-\sig(\cG_{F'})| \leq 2g - n(L,F' \#_\tau \Si)+n(L,F').$$
Noting that $n(L,F \#_\tau \Si) = n(L,F')$ and $n(L,F' \#_\tau \Si) = n(L,F)$, the above two equations combine  to give the desired conclusion.
\end{proof}

\section{Goeritz matrices and duality} \label{sec-5}
In this section, we will show  that the link invariants from the Gordon-Litherland pairing can be computed algorithmically. This is achieved by relating them to combinatorial invariants of virtual links derived from Goeritz matrices \cite{Im-Lee-Lee-2010}.

We begin with a description of the signature, determinant, and nullity invariants of checkerboard colorable virtual links due to Im, Lee, and Lee \cite{Im-Lee-Lee-2010}. The main result in this section is a duality theorem which relates the invariants of Section \ref{subsec:sign-GL}, which are defined in terms of the Gordon-Litherland pairing, with the combinatorially defined invariants of Im, Lee, and Lee, which are defined in terms of Goeritz matrices \cite{Im-Lee-Lee-2010}. As a consequence, the methods of \cite{Im-Lee-Lee-2010} give simple procedures for computing the link signatures, determinants, and nullities. These formulas are analogous to those given by Gordon and Litherland for classical links (cf.~\cite[Section 1]{GL-1978}), and as we shall see they are a direct consequence of Theorem \ref{thm:chrom-dual}.


\subsection{Tait graphs and Goeritz matrices} \label{subsec:goeritz}
In this subsection, we recall the construction of the Tait graph and Goeritz matrix associated to a checkerboard colored link in a thickened surface. We use this to define the associated signature, determinant, and nullity invariants, following \cite{Im-Lee-Lee-2010}.

We begin by recalling the construction of the Tait graph associated to a checkerboard colored link in a thickened surface.

Suppose $L \subset \Si \times I$ is a link with link diagram $D$ and checkerboard coloring $\xi$. Let $F_\xi$ be the checkerboard surface obtained from the black regions. Recall that $F_\xi$ consists of one disk for each black region and one half twisted band for each crossing. The \textit{Tait graph} is denoted $\Ga_\xi$ and is defined to be the graph in $\Si$ with one vertex for each black disk and one edge for each band. It follows that $\Ga_\xi$ is a deformation retract of $F_\xi$, alternatively $\Ga_\xi$ is the deformation retract of $\Si$ after removal of all the white disks.

Let $C_D = \{c_1,\ldots, c_k\}$ denote the set of crossings of $D$ and enumerate the white regions $X_0,X_1,\ldots, X_n$ of $\Si \sm D$. For each crossing $c \in C_D$, we define its incidence number $\eta(c) = \pm 1$ with respect to the checkerboard coloring $\xi$ according to Figure \ref{crossing-eta-type}.

Define an $(n+1) \times (n+1)$ matrix $G'_{\xi}(D)=(g_{ij})_{i,j=0,\ldots, n}$ by setting
$$g_{ij} = 
\begin{cases} 
-\sum  \eta(c) & \text{ if $i \neq j$,} \\
-\sum_{k\neq i} {g_{ik}} & \text{ if $i = j$.}
\end{cases}$$
In the above formulas, the first sum is taken over all crossings $c \in C_D$ incident to both $X_i$ and $X_j$, and the second guarantees that $\sum_{j=0}^n g_{ij} = 0$ for each $i=0, \ldots, n$. Notice that  $G'_{\xi}(D)$ is a symmetric matrix with integer entries.

\begin{definition} \label{GL-pairing}
The Goeritz matrix $G_\xi(D)$ is the $n \times n$ matrix obtained by deleting the first row and column from $G'_{\xi}(D).$ In other words, $G_\xi(D) =(g_{ij})_{i,j=1,\ldots, n}.$  
\end{definition}

The Goeritz matrix $G_\xi(D)$ is not an invariant of the link; it depends on the diagram $D$, the checkerboard coloring $\xi$, and the order of the white regions.  However, Im, Lee and Lee used this approach to define combinatorial invariants for non-split links in thickened surfaces and virtual links (cf.~\cite{Im-Lee-Lee-2010}).

Assume that $D$ is a link diagram which is checkerboard colored and connected. Define the signature, determinant, and nullity by setting
\begin{equation} \label{eqn:check-signature}
\begin{split}
\si_\xi(D) &= \sig(G_\xi(D))-\mu_\xi(D), \\
{\textstyle \det_\xi}(D)&=|\det(G_\xi(D))|, \\ 
n_\xi(D) &=\nullity(G_\xi(D)).
\end{split}
\end{equation} 
By \cite[Theorem 5.2]{Im-Lee-Lee-2010}, it follows that $\si_\xi(D), \det_\xi(D)$ and $n_\xi(D)$ give well-defined invariants of the associated link $L \subset \Si \times I$ depending only on the choice of checkerboard coloring $\xi$. (Note that our definition of the nullity $n_\xi(D)$  differs slightly from that in \cite{Im-Lee-Lee-2010}, where they define it to be equal to $\nullity(G_\xi(D))+1.$)

In general, one will get pairs of invariants. The resulting quantities are not generally invariant under stabilization. To get invariants of virtual links, one must be careful to always represent them by minimal genus diagrams.

\begin{example} Figure \ref{3-7-a} shows a checkerboard colorable knot $K$ in the thickened torus, and it admits two checkerboard colorings $\xi$ and $\xi^*$. For $\xi$, there is only one white region $X_0$, so $G'_{\xi}(D)=[0]$ and $G_{\xi}(D)$ is the empty matrix. Further, two of the crossings have type II, and $\eta(c)=-1$. Thus $\si_{\xi}(K) = \sig(G_{\xi}(D))-\mu_{\xi}(D)=0-(-2)=2$, $\det_\xi(D)=1$, and $n_\xi(K)=0$ (cf., Remark \ref{rem:convention}).

For $\xi^*$, there are two white regions $X_0,X_1$ and we compute that
$$G'_{\xi^*}(D) = \begin{bmatrix} 2 & -2 \\ -2 & 2 \end{bmatrix} \;
\text{ and } \; G_{\xi^*}(D)=\begin{bmatrix} 2 \end{bmatrix}.$$
Further, one crossing has type II, and $\eta(c) = 1.$ Thus $\si_{\xi^*}(K)=\sig(G_{\xi^*}(D))-\mu_{\xi^*}(D)=1-1=0$, $\det_{\xi^*}(D)=2$, and $n_{\xi^*}(K)=0$.
\hfill $\Diamond$ \end{example}

For checkerboard colorable virtual knots up to six crossings,  computations of the Goeritz matrices, signatures, determinants, and nullities  are available at \cite{chrisman-table}. 


\subsection{Chromatic duality}

In this subsection, we show that the signature, determinant, and nullity invariants of Section \ref{subsec:sign-GL} are equivalent to the invariants  defined in Equation \eqref{eqn:check-signature} in Section \ref{subsec:goeritz}. The first family of invariants is defined geometrically in terms of the Gordon-Litherland pairing, and the second is defined combinatorially in terms of the Goeritz matrices (cf.~\cite{Im-Lee-Lee-2010}). The correspondence between the two families of invariants is a consequence of Theorem \ref{thm:chrom-dual}, and an important aspect of the correspondence is the principle of \textit{chromatic duality}. This principle stipulates that the colorings switch from black to white or vice versa in passing from one family of invariants to the other. At first glance, this may appear to be the result of incompatible conventions, but further examination reveals that it is an intrinsic feature stemming from Theorem \ref{thm:chrom-dual}.    

Let $g =g(\Si)$ be the  genus of the surface $\Si$, and assume that the link diagram $D$ is cellularly embedded in $\Si.$ Consequently,  the inclusion map $i \co D \to \Si$ induces a surjection $i_*\co H_1(D;\ZZ) \to H_1(\Si;\ZZ).$  If $F_\xi$ is a checkerboard surface for $L$, then this implies that the map $H_1(F_\xi;\ZZ) \to H_1(\Si \times I;\ZZ)$ must also be surjective. Thus, one can find curves $\{\al_1,\ldots, \al_{n}; \ga_1,\ldots, \ga_{2g}\}$ in $F_\xi$ giving a basis for $H_1(F_\xi;\ZZ)$  such that  $\al_1,\ldots, \al_{n}$ lie in the kernel of $H_1(F_\xi;\ZZ)\to H_1(\Si;\ZZ)$ and $\ga_1,\ldots, \ga_{2g}$ map to a set of generators for $H_1(\Si;\ZZ)$.
 
\begin{lemma} \label{lemma:basis-conv}
Suppose $D$ is a checkerboard colorable link diagram on $\Si$ with coloring $\xi$ and associated checkerboard surface $F_\xi$. Then there is a basis $\{ \al_1,\ldots, \al_{n}; \ga_1,\ldots, \ga_{2g} \}$ for $H_1(F_\xi;\ZZ)$, such that $\al_1,\ldots, \al_{n}$ lie in the kernel of $H_1(F;\ZZ)\to H_1(\Si;\ZZ)$, and the matrix representative of the pairing $\cG_{F_\xi} \co H_1(F_\xi;\ZZ) \times H_1(F_\xi;\ZZ) \to \ZZ$ on the subset $\{\al_1,\ldots, \al_{n}\}$ is the Goeritz matrix $G_{\xi}(D) = (g_{ij})$ of Definition \ref{GL-pairing}.
\end{lemma}

\begin{proof}
Let $\Ga_\xi$  be the Tait graph of $F_\xi$, this is the graph in $F_\xi \cap (\Si \times \{1/2\})$ with one vertex for each disk and one edge for each band. Notice that $\Ga_\xi$ is a deformation retract of $F_\xi$, alternatively $\Ga_\xi$ is the deformation retract of $\Si$ after removal of all the white disks. Let $\ga_1,\ldots, \ga_{2g}$ be a set of standard generators for $H_1(\Si;\ZZ)$, and  label the regions of $\Si \sm \Ga_\xi$ as $Y_0,Y_1,\ldots, Y_n$, so that $Y_i$ contains $X_i$ for $i=0,\ldots, n$. Each $Y_i$ is oriented using the orientation of $\Si$. Set $\al_{i} = [\partial Y_i],$ obtaining homology classes which, together with  $\ga_1,\ldots, \ga_{2g}$, generate $H_1(\Ga_\xi;\ZZ) \cong H_1(F_\xi;\ZZ)$ and with just one relation $\sum_{i=0}^n \al_i = 0.$

\begin{figure}[ht]
\centering
\includegraphics[scale=0.90]{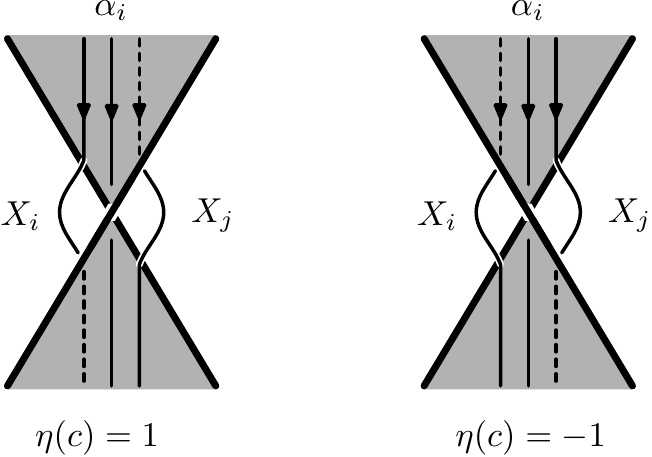} 
\caption{Double points $c$ contribute with sign $\eta(c)$.}
\label{crossing-eta-2}
\end{figure}

Since  $\sum_{i=0}^n \al_i = 0$ and  $\sum_{j=0}^n g_{ij} = 0$, it is enough to show that 
$$\cG_{F_\xi}(\al_i,\al_j) = g_{ij} \text{ for $i \neq j$.}$$
Since $p_*(\al_i) \cdot p_*(\al_j) = 0$ for all $i, j$, we have $\cG_{F_\xi}(\al_i,\al_j) = \lk_\Si(\tau \al_i,\al_j)= \tau \al_i\cdot Y_j.$ Notice that $\tau \al_i$ intersects $Y_j$ only near double points incident to both $X_i$ and $X_j$, and each such double point $c$ contributes with sign $\eta(c)$, see Figure \ref{crossing-eta-2}.
\end{proof}

Suppose $F_\xi$ is a checkerboard surface associated to a checkerboard coloring $\xi$, and let $F' = F_\xi \#_\tau \Si$ be the spanning surface obtained from attaching a parallel copy of the Carter surface near $\Si \times \{0\}$ to $F_\xi$ by a small thin tube $\tau$ (see Section \ref{section:S-star-equivalence}). 

Notice that the rank of $H_1(F';\ZZ)$ is equal to the rank of $H_1(F_\xi;\ZZ) + 2g.$ Proposition \ref{prop-S*} implies that  $F'$ is $S^*$-equivalent to the chromatic dual $F_{\xi^*}$. The next result relates the signature, determinant, and nullity invariants coming from the Gordon-Litherland pairing ${\cG}_{F'}$ (see Equations \eqref{defn:signature} and \eqref{defn:det}) to those defined in Equation \eqref{eqn:check-signature} in terms of the Goeritz matrices for the $\xi$ coloring. It is an immediate consequence of Theorem \ref{thm:chrom-dual} and Lemma \ref{lemma:basis-conv}.

\begin{theorem}\label{thm:chromatic} Given a checkerboard colorable diagram $D$ with coloring $\xi$ and $F' = F_\xi \#_\tau \Si$ as above, the signature, determinant, and nullity invariants of the Gordon-Litherland pairing ${\cG}_{F'}$  are equal to those defined using the Goeritz matrices of its chromatic dual. In particular, we have
$$\si(D,F') = \si_{\xi}(D), \; \det(D,F')= {\textstyle \det_{\xi}}(D), \text{ and } n(D,F') = n_{\xi}(D).$$
\end{theorem}

Switching the roles of the checkerboard surfaces, it follows that
$$\sigma(D,F_\xi) = \sigma_{\xi^*}(D), \; \det(D,F_\xi) = {\textstyle \det_{\xi^*}}(D), \text{ and } n(D,F_\xi) = n_{\xi^*}(D).$$
This again is the principle of \textit{chromatic duality}.

\subsection{Crossing change}
In this subsection, we study the effect on the signature $\si_\xi(L)$ of changing a crossing of a link $L$ in a thickened surface.

A well-known result for classical knots implies that, under crossing change,  the signature changes by at most two \cite{Murasugi-1965}. In that same paper,  Murasugi studied the relationship between the signatures and nullities for links related by smoothing a crossing. The following result gives a generalization for checkerboard colorable links in thickened surfaces. 

The idea of the proof is similar to Murasugi's original argument. It involves applying Equation \eqref{sig-formula} to analyze how the signature of a checkerboard colored link changes under a crossing change. Notice that the checkerboard coloring depends only on its projection $p(L) \subset \Si$, where $p \colon \Si \times I \to \Si$. So we can use the same coloring for links related by a crossing change. 

\begin{theorem} \label{thm:cc}
Let $D_{+}$ and $D_-$ be two checkerboard colorable link diagrams on a surface $\Si$ which are identical everywhere except at one crossing, which is positive for $D_+$ and negative for $D_-.$ If $\xi$ is a checkerboard coloring for $D_+$ (and $D_-$), then the signatures satisfy
$$\si_\xi(D_+) \leq \si_\xi(D_-) \leq \si_\xi(D_+)+2.$$
Indeed, there are two cases, according to the nullities $n_\xi(D_\pm)$. 
\begin{enumerate}
\item[(i)] If $n_\xi(D_+)=n_\xi(D_-)$, then either $\si_\xi(D_-) = \si_\xi(D_+)$ or $\si_\xi(D_-)=\si_\xi(D_+)+2.$
\item[(ii)] If $n_\xi(D_+) \neq n_\xi(D_-)$, then $\si_\xi(D_-) = \si_\xi(D_+)+1.$
\end{enumerate}
\end{theorem}

\begin{proof}
Let $c_+$ denote the distinguished crossing in $D_{+}$, and $c_-$ the corresponding crossing of $D_-$. So $c_+$ is a positive crossing and $c_-$ is negative. There are two cases according to the value of $\eta(c_+)  \in\{ \pm 1\}.$  The proofs for the two cases are similar, so we give the first and leave the second to the reader. 

Therefore assume that $\eta(c_+) = - 1$. Then $\eta(c_-)=1$, and $c_+$ and $c_-$ are both type I crossings. Further, the correction terms satisfy $\mu_\xi(D_+)= \mu_\xi(D_-)$ and the Goeritz matrices are related as follows:

$$G_{\xi}(D_{+})=\left[\begin{matrix}
x&U\\ U^{\tr}& *
\end{matrix}\right] \; \text{ and } \;
G_{\xi}(D_{-})=\left[\begin{matrix}
x+2&U\\ U^{\tr}& *
\end{matrix}\right].$$  

There are two cases according to the nullities $n_\xi(D_\pm)$.

\smallskip
\noindent
{\it Case I:} $n_{\xi}(D_{+})=n_{\xi}(D_{-})$. Then the rank of $G_{\xi}(D_{+})$ is equal to the rank of $G_{\xi}(D_{-})$, and we can choose chains of principal minors $M_i^\pm$ for $G_{\xi}(D_\pm)$ as in Equation \eqref{sig-formula} so that $M_i^+ = M_i^-$ for $i=0,\ldots, r-1$. If the submatrix for the $r$-th minor $M^\pm_{r}$ does not contain the upper left hand entry (which is $x$ for $G_{\xi}(D_{+})$ and $x+2$ for $G_{\xi}(D_{-})$), then $M^+_r = M^-_r$, and Equation \eqref{sig-formula} implies that $\sig(G_{\xi}(D_{+}))=\sig(G_{\xi}(D_{-}))$. Hence $\si_{\xi}(D_{+})=\si_{\xi}(D_{-})$.

Otherwise, if $M^+_{r}$ and $M^-_r$ have the same sign, then Equation \eqref{sig-formula} again implies that $\sig(G_{\xi}(D_{+}))=\sig(G_{\xi}(D_{-}))$, and $\si_{\xi}(D_{+})=\si_{\xi}(D_{-})$.  On the other hand, if $M^+_{r}$ and $M^-_r$ have opposite signs, then since the upper left hand entry of $G_{\xi}(D_{-})$ is $2$ more than the corresponding entry for $G_{\xi}(D_{+})$, Equation \eqref{sig-formula} implies that $\sig(G_{\xi}(D_{-}))=\sig(G_{\xi}(D_{+}))+2$, so $\si_{\xi}(D_{-})=\si_{\xi}(D_{+})+2$.

\smallskip
\noindent
{\it Case  II:} $n_{\xi}(D_{+})\neq n_{\xi}(D_{-})$. Then $|n_{\xi}(D_{+})- n_{\xi}(D_{-})|=1$.  Suppose firstly that $G_{\xi}(D_{-})$ has rank $r+1$ and $G_{\xi}(D_{+})$ has rank $r.$ We can choose chains of principal minors $M_i^\pm$ for $G_{\xi}(D_\pm)$ as in Equation \eqref{sig-formula} so that $M_i^+ = M_i^-$ for $i=0,\ldots, r$. Notice that since $G_{\xi}(D_{-})$ is the matrix with larger rank, $M^-_{r+1}$ will contain the upper left hand entry of $G_{\xi}(D_{-})$. Also, $\det G_{\xi}(D_{+}) =0$, and $M^-_{r}M^-_{r+1}$ will be larger than $M^+_{r}M^+_{r+1}=0$. As a result $\sig(G_{\xi}(D_{-}))=\sig(G_{\xi}(D_{+}))+1$, so $\si_{\xi}(D_{-})=\si_{\xi}(D_{+})+1$. The same formula can be established in the case when  $G_{\xi}(D_{+})$ has rank $r+1$ and $G_{\xi}(D_{-})$ has rank $r$ using a similar argument.
\end{proof}
 
\subsection{Mirror Images}

In this subsection, we will relate the signature, determinant, and nullity invariants of a checkerboard colorable link in a thickened surface to those of its mirror images.  

For links in thickened surfaces, there are two ways to take the mirror image, one is called the \textit{vertical} mirror image and the other is called the \textit{horizontal} mirror image. This terminology is consistent with the terminology commonly used for mirror images of virtual links, see \cite{green}.

\begin{definition} Let $L \subset \Si \times I$ be an oriented link in a thickened surface. \begin{itemize}
\item[(i)] Consider the orientation reversing homeomorphism $\varphi \co \Si \times I \to \Si \times I$ given by $\varphi(x,t) =(x,1-t)$. The image of $L$ under $\varphi$ is called the \textit{vertical mirror image} of $L$ and is denoted $L^{*}$. A diagram of $L^{*}$ is obtained  from a diagram of $L$ by switching all the crossings.

\item[(ii)] Let $f\co \Si \to \Si$ be an orientation reversing homeomorphism and set $\psi\co \Si \times I \to \Si \times I$ to be an orientation reversing homeomorphism given by $\psi(x,t) =(f(x),t)$. The image of $L$ under $\psi$ is called the \textit{horizontal mirror image} of $L$ and is denoted $L^{\dag}$. A Gauss diagram of $L^{\dag}$ is obtained by changing the sign on every arrow in a Gauss diagram of $L$.
\end{itemize}
\end{definition} 
 
Let $L\subset \Si \times I$ be a link with spanning surface $F\subset \Si \times I$, and let $F_\varphi$ and $F_\psi$ be the surfaces obtained by taking the images of $F$ under the maps $\varphi$ and $\psi,$ respectively. Then $F_\varphi$ is a spanning surface for $L^*$ and $F_\psi$ is a spanning surface for $L^\dag.$

The next result relates the signatures, determinants, and nullities of $(L^*,F_\varphi)$ and $(L^\dag, F_\psi)$ to those for $(L,F)$. The proof is standard, and we provide it for the reader's convenience. For an alternative approach, see  \cite[Theorem 2.10]{Karimi-kh}.

\begin{proposition}\label{prop:mirror_image}
Let $L\subset \Si \times I$ be a  link with spanning surface $F \subset \Si \times I$. Then the signature, determinant, and nullity of the vertical and horizontal mirror images of $L$ satisfy
\begin{eqnarray*}
\si(L^*, F_\varphi)=-\si(L,F)  & \text{and} & \si(L^{\dag},F_\psi)=-\si(L,F), \\ 
\det(L^*,F_\varphi)=\det(L,F)  &\text{and}& \det(L^{\dag},F_\psi)=\det(L,F), \\
n(L^*,F_\varphi)=n(L,F)  &\text{and}& n(L^{\dag},F_\psi)=n(L,F).
\end{eqnarray*}

\end{proposition}

\begin{proof}
Let $\{\al_1,\dots,\al_n\}$ be a basis for $H_1(F;\ZZ)$, then $\{\varphi(\al_1), \dots,\varphi(\al_n)\}$ is a basis for $H_1(F_\varphi;\ZZ)$. We compute that 
$$\lk_\Si(\tau \varphi(\al),\varphi(\be)) = \lk_\Si(\varphi(\tau\al),\varphi(\be))=-\lk_\Si(\tau\al,\be).$$
The second step results from the fact that $\varphi$ is an orientation reversing homeomorphism. It follows that $\cG_{L^*,F_\varphi} = - \cG_{L,F},$ and that $\sig(\cG_{L^*,F_\varphi}) = - \sig(\cG_{L,F}).$

On the other hand, if $L'$ is a parallel of $L$ that misses $F$, then $\varphi(L')$ is a parallel of $\varphi(L)$ that misses $F_\varphi$. Thus,  $\lk_\Si(\varphi(L),\varphi(L')) = - \lk_\Si(L,L')$, and $e(L^*,F_\varphi) = -e(L,F).$ The formulas for $\si(L^*, F_\varphi), \det(L^*,F_\varphi)$ and $n(L^*,F_\varphi)$ now follow directly. A similar argument gives the stated formulas for $\si(L^{\dag},F_\psi),\det(L^{\dag},F_\psi),$ and $n(L^{\dag},F_\psi)$.  
\end{proof}

\subsection{Almost classical links}
In this subsection, we consider almost classical links. We relate the signature, determinant and nullity invariants defined using the Gordon-Litherland pairing to the signature, determinant and nullity invariants defined via the Seifert pairing.

To begin, we show that every almost classical link admits a checkerboard colorable diagram whose checkerboard surface $F_\xi$ is oriented.

\begin{proposition}
If $L$ is an almost classical link, then it can be represented by a  diagram $D$ on a minimal genus surface $\Si$ with a checkerboard coloring $\xi$, so that every crossing has type I. Thus, the checkerboard surface $F_\xi$ is oriented.
\end{proposition}

\begin{proof}
Since $L$ is almost classical, it can be represented as a homologically trivial link in a thickened surface. If the surface is not minimal genus, then perform a destabilization, and notice that the link $L$ is still represented by a homologically trivial link in the destabilized surface. Thus, after a finite sequence of destablizations, it follows that $L$ can be represented by a homologically trivial link on a surface $\Si$ of minimal genus. If $D$ is the resulting diagram on $\Si$ for $L$, then Proposition \ref{prop:equiv} implies that $D$ is checkerboard colorable. 

Since $L$ is homologically trivial, we have  a Seifert surface $F$ for $L$ in $\Si \times I$. The surface $F$ can be realized as a union of disks and bands. Performing an isotopy of $F$, we can  shrink the disks so their images under projection $\Si \times I \to \Si$ are disjoint from one another and also disjoint from each band. The isotopy of $F$ induces an isotopy of the link diagram, and notice that the new link diagram may no longer be a minimal crossing diagram for $L$. 

Our goal is to show that this new diagram can be isotoped further so that the Seifert surface coincides with the spanning surface associated to the black regions. (This is equivalent to showing that $L$ can be represented by a special diagram in the sense of \cite[Definition 13.14]{Burde-Zieschang-Heusener}.) By construction, the new diagram bounds a Seifert surface, which projects one-to-one under $\Si \times I \to \Si$ except possibly at the intersections of the bands. 

\begin{figure}[ht]
\centering
\includegraphics[scale=1.10]{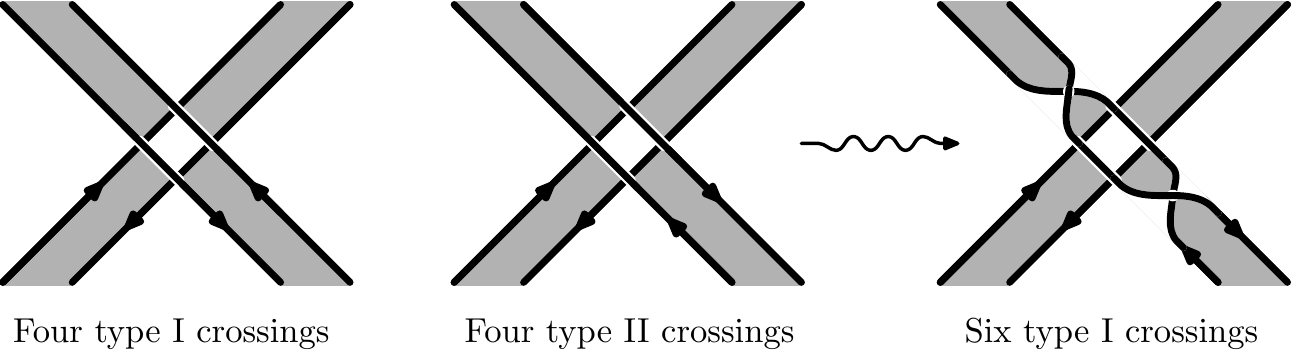} 
\caption{Band intersections with all crossings of the same type. An isotopy of the middle diagram gives the one on the right.}
\label{crossing-band}
\end{figure}

Whenever two bands intersect, the four crossings all have the same type, see the two diagrams on the left of Figure \ref{crossing-band}. If the four crossings have type II, then one can perform a Reidemeister 2 move to make them type I crossings, see the diagram on the right of Figure \ref{crossing-band}. After performing a finite sequence of such moves, the new link diagram will have only type I crossings. Consequently, the black regions of its associated checkerboard coloring will form an oriented spanning surface, and this completes the proof of the proposition.
\end{proof}

If the checkerboard surface $F_\xi$ is oriented, then so is the surface $F_\xi \#_\tau \Si$ obtained by tubing off  a parallel copy of the Carter surface. The next result follows from our previous observation that $F_\xi \#_\tau \Si$ is $S^*$-equivalent to the dual surface $F_{\xi^*}$, and Lemma \ref{lemma:S*-inv}, which shows that the checkerboard signatures and nullities are invariant under $S^*$-equivalence.
\begin{corollary} \label{cor-ac1}
Let $L$ be an almost classical link represented by a minimal genus diagram $D$ with checkerboard coloring $\xi$ whose checkerboard surface $F_\xi$ is oriented. Then the  signatures $\si(L,F_\xi)$ and $\si(L, F_\xi \#_\tau \Si)$ of the Seifert matrices are equal to the checkerboard signatures  $\si_\xi(D)$ and $\si_{\xi^*}(D)$, and the nullities $n(L,F_\xi)$ and $n(L,F_{\xi} \#_\tau \Si)$ are equal to the checkerboard nullities  $n_\xi(D)$ and $n_{\xi^*}(D)$.
\end{corollary}

The following result is a direct consequence of Corollary \ref{cor-ac1}, and it summarizes the situation for almost classical links.

\begin{corollary} \label{cor-ac2}
Given an almost classical link $L$ with Seifert surface $F$, the signature $\si(L,F)$  is equal to the checkerboard signature $\si_\xi(D)$ for some coloring $\xi$. Conversely every checkerboard signature $ \si_\xi(D)$ is equal to the signature $\si(L,F)$ for some Seifert surface $F$.
\end{corollary}

As a consequence of \cite[Theorem 2.5]{Boden-Chrisman-Gaudreau-2017a}, it follows that for almost classical knots, the checkerboard signatures $\si_\xi(K)$ are slice obstructions and give information on the slice genus of $K$. (Definitions of virtual concordance for knots in thickened surfaces and virtual knots can be found in \cite{Boden-Chrisman-Gaudreau-2017a}.) It is an interesting problem to extend those concordance results to all checkerboard colorable knots. We hope to address that question in future research.


\section{Branched covers and intersection forms} \label{sec-6}
In this section, we relate the Gordon-Litherland pairing to the relative intersection form of a certain double branched cover of $W \times I$. In order to do that, we recall some background material on intersection forms for 4-manifolds with boundary.


\subsection{Relative intersection forms}  
Let $X$ be a compact, connnected, oriented $4$-manifold with $\partial X \ne \varnothing$. A \textit{decomposing pair} of $\partial X$ is a pair of compact $3$-manifolds $(Y,Y')$ with boundary such that $Y \cup Y'=\partial X$, $Y \cap Y'= \partial Y=\partial Y'$. In the following, we will consider two decomposing pairs $(Y,Y')$ and $(Z, Z')$ such that $Y \subset \Int(Z')$ and $Z \subset \Int(Y')$. We will refer to $(Y,Y'), (Z,Z')$ as \textit{dual boundary decompositions} of $\partial X.$ Notice that for all dual boundary decompositions, we have that $Y \cap Z=\varnothing$ and $Y' \cup Z'=\partial X$.

\begin{example} \label{ex_top_bot_thick} Let $W$ be a compact oriented $3$-manifold such that $\partial W \ne \varnothing$ and set $X=W \times I$. Let $Y=W \times \{0\}$ and $Y'=\partial W \times I \cup W \times \{1\}$, and likewise, let $Z=W \times \{1\}$ and $Z'=\partial W \times I \cup W \times \{0\}$. Then  $(Y,Y')$ and $(Z,Z')$ are dual boundary decompositions of $\partial (W \times I)$.
\hfill $\Diamond$ \end{example}

\begin{example} \label{ex_top_bot_triv} The \textit{trivial dual boundary decompositions} of $X$ are given by setting $Y=Z=\varnothing$ and $Y'=Z'=\partial X$. 
\hfill $\Diamond$ \end{example}

Let $[X] \in H_4(X,\partial X;\ZZ)$ denote the fundamental class of $X$. Then Poincar\'{e} duality for manifolds with boundary implies that cap product with $[X]$ gives isomorphisms (see \cite[Theorem 3.43]{Hatcher-2002}  or \cite[p.358]{Bredon-1993}):
\begin{align*}
&\cdot \smallfrown [X] \co H^2(X,Y';\ZZ) \lto H_2(X,Y;\ZZ), \\ 
&\cdot \smallfrown [X] \co H^2(X,Z';\ZZ) \lto H_2(X,Z;\ZZ). 
\end{align*} 

Let $D_0\co H_2(X,Y;\ZZ)\lto H^2(X,Y';\ZZ)$ and $D_1\co H_2(X,Z;\ZZ)\lto H^2(X,Z';\ZZ)$ denote the inverses of these isomorphisms, respectively. For $\al \in H^2(X,Y';\ZZ)$, $\be \in H^2 (X,Z';\ZZ)$, observe that:
$$\al \smallsmile \be \in H^4(X,Y' \cup Z';\ZZ)=H^4(X,\partial X;\ZZ) \cong H_0(X;\ZZ) \cong \ZZ,$$
where $\cdot \smallsmile \cdot$ denotes the relative cup product. With these definitions in place, we now define a relative intersection form for $X$.

\begin{definition}  
For dual boundary decompositions $(Y,Y')$ and $(Z,Z')$ of $\partial X$, the \textit{relative intersection form} is the pairing:
$$\cQ \co H_2(X,Y;\ZZ) \times H_2(X,Z;\ZZ) \lto \ZZ, $$
given by setting $\cQ(\al,\be)=\langle D_0(\al) \smallsmile D_1(\be), [X] \rangle.$
\end{definition}

\begin{example} For the dual boundary  decompositions of $\partial X$ for $X=W\times I$ in Example \ref{ex_top_bot_thick}, $H_2(X,Y;\ZZ) \cong H_2(X,Z;\ZZ) \cong 0$ and $\cQ\equiv 0$.
\hfill $\Diamond$ \end{example}

\begin{example} For the trivial dual boundary decompositions of $\partial X$ in Example \ref{ex_top_bot_triv} with $Y=Z=\varnothing$, the relative intersection form is identical to the usual intersection form on $X$.
\hfill $\Diamond$ \end{example}

Suppose $\al \in H_2(X,Y;\ZZ)$ and $\be \in H_2(X,Z;\ZZ)$ are represented by compact oriented surfaces $F,G$ smoothly embedded in $X$ with $\partial F \subset Y$ and $\partial G \subset Z$. Assuming that $F$ and $G$ intersect transversely in $X$, then $F \cap G$ consists of a finite set of points in $\Int(X)$. In this case, we have the relative intersection pairing:
$$\xymatrix{
H_2(F,\partial F;\ZZ) \otimes H_2(G,\partial G;\ZZ) \ar[r]^-{\bigcdot} & H_0(F \cap G;\ZZ) \ar[r]^-{\ep} & \ZZ,} $$  
where $\ep \co H_0(F \cap G;\ZZ) \to \ZZ$ is the augmentation map. The composition $\ep(\al\bigcdot \be)$ can be calculated as the signed sum of the local intersection numbers of $F \cap G$. Furthermore, it can be shown that $\ep(\al \bigcdot \be)=\langle D_0(\al)\smallsmile D_1(\be),[X] \rangle$ (see \cite[Chapter VIII, Section 13]{Dold-1980}). 


\subsection{Mirror double branched covers} \label{sec_double} 

Given a 4-manifold of the form $W\times I$ and a compact surface $F$ in $W
\times I$, we construct the mirror double cover $W\times I$ branched along $F$. We will use this construction to show the Gordon-Litherland pairing is equivalent to the relative intersection form of the mirror double branched cover. 

To begin, we recall the relevant results for classical knots. Let $K$ be a knot in $S^3$ with spanning surface $F$, and let $M_{\wh{F}}$ be the 2-fold cover of $B^4$ branched along $\wh{F}$, a copy of $F$ with $\Int(F)$ pushed into $\Int(B^4)$. Gordon and Litherland \cite{GL-1978} showed that there is an isometry between $(H_2(M_{\wh{F}};\ZZ), \, \bigcdot \,)$ and $(H_1(F;\ZZ), \cG_F)$, where $\bigcdot$ denotes the intersection form on $H_2(M_{\wh{F}};\ZZ)$. 

We now explain how to generalize these results to checkerboard colorable knots $K \subset \Si \times I$.

Let $F\subset\Si \times I$ be a compact connected surface with $\partial F=K$. Suppose $W$ is a compact oriented $3$-manifold with $\partial W=\Si$. Push $\Int (F)$ into $\Int(W \times I)$ to obtain a properly embedded surface $\wh{F} \subset W \times I.$  The mirror double cover branched along $\wh{F}$ is denoted $M_{\wh{F}}$ and constructed as follows.

First, cut the 4-manifold $W \times I$ open along the trace of the isotopy which pushes $\Int (F)$ into $\Int(W \times I)$. The cut parts are homeomorphic to a tubular neighborhood $N$ of $F$ in $\Si \times I$, which is an $I$-bundle over $F$. Consider a second diffeomorphic copy of the cut 4-manifold under the map $W \times I \to W \times I$ sending $(x,t)\mapsto (x,1-t).$ Notice that the two copies are diffeomorphic by an orientation reversing diffeomorphism.

We will use $W_1, W_2$ to denote the two copies of $W$ and $N_1 \subset W_1 \times I$
and $N_2 \subset W_2 \times I$ for the tubular neighborhoods of $F$ in the two copies. Set
$$M_{\wh{F}}=(W_1\times I)\cup(W_2\times I)/((x,t)\in N_1\subset W_1\times I\sim (x,1-t)\in N_2\subset W_2\times I).$$
We obtain dual boundary decompositions of $\partial M_{\wh{F}}$ by setting 
$$Y = (W_1\times \{0\})\cup(W_2\times \{0\}), \quad Z = (W_1\times \{1\})\cup (W_2\times \{1\}),$$ 
and taking the obvious choices for $Y', Z'$. Then we have a relative intersection form:
$$\cQ_F\co H_2(M_{\wh{F}},Y;\ZZ) \times H_2(M_{\wh{F}},Z;\ZZ) \to \ZZ.$$

To identify $\cQ_F$ with the Gordon-Litherland form $\cG_F$, first apply Mayer-Vietoris to the pairs $(M_{\wh{F}},Y)$ and $(M_{\wh{F}},Z)$. The connecting homomorphisms give isomorphisms:
\begin{equation} \label{eqn:isomorphisms}
\begin{split}
&\varphi_0 \co H_2(M_{\wh{F}},Y;\ZZ) \to H_1(N;\ZZ)\cong H_1(F;\ZZ), \\
&\varphi_1 \co H_2(M_{\wh{F}},Z;\ZZ) \to H_1(N;\ZZ)\cong H_1(F;\ZZ).
\end{split}
\end{equation}
This can be seen by writing $M_{\wh{F}}$ as a union $(W_1\times I)\cup (W_2\times I)$ and noting that $H_n(W_i\times I, W_i\times \{1\};\ZZ)\cong H_n(W_i\times I, W_i\times \{0\};\ZZ)\cong 0$ for  $i=1,2$ and $n \geq 0$. 

\begin{theorem} \label{thm:isom} 
For $a,b \in H_1(F;\ZZ)$, we have $\cG_F(a,b)=\cQ_F(\varphi_0^{-1}(a), \varphi_1^{-1}(b))$.
\end{theorem}

\begin{proof} The inverse maps of $\varphi_0$ and $\varphi_1$ may be described as follows. For the tubular neighborhood $N$ of $F$ in $\Si \times I$ as above, let $i \co F \hookrightarrow N$ be the inclusion map. Since $N$ is an $I$-bundle over $F$, it follows that $i_* \co H_1(F;\ZZ) \to H_1(N;\ZZ)$ is an isomorphism. 

Suppose $\al\subset N$ is a simple closed curve. For $i=1,2$ and $j=0,1,$ let $\Cyl_{i,j} (\al)$ be a surface in $W_i\times I$, connecting $\al \subset N_i$ to a simple closed curve in $W_i\times \{j\}$. Define: 
\begin{equation*}
\begin{split}
C_0(\al) &= \Cyl_{1,0}(\al)\cup (-\Cyl_{2,0}(\al)), \\ 
C_1(\al) &= \Cyl_{1,1}(\al)\cup (-\Cyl_{2,1}(\al)).
\end{split}
\end{equation*}
Then $C_0(\al), C_1(\al)$  represent relative 2-cycles in $H_2(M_{\wh{F}},Y;\ZZ),H_2(M_{\wh{F}},Z;\ZZ)$, respectively. Set $C_0([\al]) = [C_0(\al)]$, $C_1([\al])=[C_1(\al)]$. It follows from the Mayer-Vietoris sequences defining $\varphi_0,\varphi_1$ that $C_0=\varphi_0^{-1}$ and $C_1=\varphi_1^{-1}$. 

If $\al,\be$ are two disjoint 1-cycles in $N$, then  
$$\Cyl_{1,0}(\al) \cap \Cyl_{2,1}(\be) = \varnothing = \Cyl_{2,0}(\al) \cap \Cyl_{1,1}(\be).$$
Therefore, the intersection of the homology classes is given by:
\begin{equation*}
\begin{split}
C_0([\al])\bigcdot C_1([\be])
&= \left( \Cyl_{1,0}(\al)\cup (-\Cyl_{2,0}(\al))\right) \bigcdot \left(\Cyl_{1,1}(\be)\cup (-\Cyl_{2,1}(\be)) \right) \\
&= \Cyl_{1,0}(\al) \bigcdot \Cyl_{1,1}(\be)  + \Cyl_{2,0}(\al) \bigcdot \Cyl_{2,1}(\be).
\end{split}
\end{equation*}

Now suppose $a,b \in H_1(F;\ZZ)$ with $i_*(a), i_*(b)$ represented by cycles $\al, \be$ in $N$. It is then clear from the construction that:
$$\ep(C_0([\al])\bigcdot C_1([\be]))=\overline{\lk}_{\Si}(a,b)+\underline{\lk}_{\Si}(a,b)=\overline{\lk}_{\Si}(a,b)+\overline{\lk}_{\Si}(b,a).$$
This uses linking $\underline{\lk}_{\Si}(a,b)$ relative to the bottom in the second copy of $W \times I$, because of the orientation reversing diffeomorphism $(x,t) \to (x,1-t)$ on that component. Note that the last step follows from the fact that $\underline{\lk}_{\Si}(a,b)  = \overline{\lk}_{\Si}(b,a)$.

Notice that $\al$ and $\tau \be$ are disjoint cycles in $N$, as are $\be$ and $\tau \al$, and that $\tau \al, \tau \be$ are homologous in $N$ to $2\al, 2\be$, respectively. These observations together with the above equation show that
\begin{equation*}
\begin{split}
\cQ_F(\al,\be)&= \ep(C_0([\al]) \bigcdot C_1([\be]) \\
&= \tfrac{1}{4}  \ep(C_0([\al])\bigcdot C_1([\tau \be])) + \tfrac{1}{4} \ep( C_0([\tau \al]) \bigcdot C_1([\be]))  \\
&= \tfrac{1}{4}\left( \overline{\lk}_{\Si}(\al, \tau \be) + \overline{\lk}_{\Si} (\tau \be, \al) + \overline{\lk}_{\Si} (\tau \al, \be) + \overline{\lk}_{\Si}(\be, \tau \al) \right) \\
&= \tfrac{1}{4}\left( 2 \overline{\lk}_{\Si} (\tau \be, \al) 
- p_*(\tau \be) \cdot p_*(\al)  + 2\overline{\lk}_{\Si} (\tau \al, \be) - p_*(\tau \al) \cdot p_*(\be) \right) \\
&= \tfrac{1}{2}\left(  \overline{\lk}_{\Si} (\tau \be, \al) 
- p_*(\be) \cdot p_*(\al)  + \overline{\lk}_{\Si} (\tau \al, \be) - p_*(\al) \cdot p_*(\be) \right) \\
&= \tfrac{1}{2} \left(\cG_F(\be, \al) + \cG_F(\al, \be) \right) \\
&= \cG_F(\al, \be).
\end{split}
\end{equation*}
The third step requires one to apply Equation \eqref{eq:linking} to the first and last terms in line two, and the last step uses Equation \eqref{defn:GL} and the fact that $\cG_F$ is symmetric, cf., Lemma \ref{lem:sym}.  
\end{proof}


\subsection{Relative handlebody decompositions} 
Any $4$-dimensional handlebody without 1- or 3-handles can be described as a surgery on a framed link, and in that case the intersection form of the 4-manifold is represented by the linking matrix \cite[Proposition 4.5.11]{Gompf-Stipsicz}. We will develop  analogous results for relative 2-handlebodies, which are $4$-manifolds obtained by attaching 2-handles to $W \times I$, where $W$ is a compact oriented $3$-manifold with $\partial W =\Si$.

Given a knot $K$ in $\Si \times I$, a \textit{framing} is a choice of parallel $K'$ to $K$. The 0-framing is the parallel with $\lk_{\Si}(K,K')=0.$ (Note that this implies $\lk_{\Si}(K',K)=0$ as well. In fact, although the linking pairing $\lk_{\Si}(\cdot, \cdot)$ is not generally symmetric, if $K,K'$ are parallel curves, then they cobound a ribbon in $\Si \times I,$ and $\lk_{\Si}(K,K')$ and $\lk_{\Si}(K',K)$ are both equal to the number of full twists in that ribbon.) For $\ka \in \ZZ,$ the $\ka$-framing is obtained from the $0$-framing by adding $|\ka|$ full twists, where we use right-hand twists if $\ka$ is positive and left-hand twists if $\ka$ is negative. If $K$ is a framed knot, we use $\fr(K)=\ka$ to denote its framing. A framed oriented link $L=K_1 \cup \cdots \cup K_n \subset \Si \times I$ is then an oriented link with a choice of framing for each component $K_i$. The \textit{linking matrix} of $L$ is the $n \times n$ matrix whose $(i,j)$ entry is equal to $\lk_{\Si}(K_i,K_j)$ for $i \ne j$ and to $\fr(K_i)$ for $i=j$.

Given a framed link $L=K_1 \cup \cdots \cup K_n\subset\Si \times I$ and compact oriented 3-manifold $W$ with $\partial W =\Si$, let $X$ be a $4$-manifold obtained by attaching 2-handles $h_1^2,\ldots,h_n^2$ to $W \times I$ along $L$ so that $K_i$ is the attaching sphere for $h_i^2$.

We define a relative intersection form on $X$ as follows. Let $Y=W\times \{0\}, Y'=\overline{\partial X\smallsetminus Y}$ and $Z=W \times \{1\}, Z'=\overline{\partial X\smallsetminus Z}$. It is straightforward to check that $(Y,Y')$ and $(Z,Z')$ give  dual boundary decompositions on $X$. The relative intersection form on $X$ is given by 
$$\cQ_X\co H_2(X,Y;\ZZ) \times H_2(X,Z;\ZZ) \to \ZZ.$$

Fix an orientation of $L$. Let $S_1,\ldots, S_n$ be surfaces in $\Si \times I$ such that $\partial S_i=K_i \cup (\bigcup_j \ga_{i,j})$, where $\{\ga_{i,j}\}$ is a collection of closed curves in $\Si \times \{0\}$. Set $\be_{i}=\text{core}(h_i^2) \cup -S_{i}$. Then $B_Y=\{\be_{1},\ldots,\be_{n}\}$ gives a basis for $H_2(X,Y;\ZZ)$. Likewise, let $S'_1,\ldots, S'_n$ be surfaces in $\Si \times I$ such that $\partial S'_i=K_i \cup (\bigcup_j \ga'_{i,j})$, where $\{\ga'_{i,j}\}$ is a collection of closed curves in $\Si \times \{1\}$. Set $\be'_{i}=\text{core}(h_i^2) \cup -S'_{i}$. Then  $B_Z=\{\be'_{1},\ldots,\be'_{n}\}$ is a basis for $H_2(X,Z;\ZZ)$.
 
\begin{proposition} \label{prop_link_matrix_rel_int} 
The linking matrix of $L$ represents the relative intersection form $\cQ_X \co H_2(X,Y;\ZZ) \times H_2(X,Z;\ZZ) \to \ZZ$ with respect to the bases $B_Y, B_Z$.
\end{proposition}
\begin{proof} Clearly, $H_2(X,W \times I;\ZZ)$ is a free abelian group generated by the cores of the $2-$handles $h_1^2, h_2^2, \ldots, h_n^2$ (see \cite[Chapter V.4]{Dold-1980}). Further, since $H_2(X,Y;\ZZ)\cong H_2(X,W \times I;\ZZ) \cong H_2(X,Z;\ZZ)$, it follows that  $H_2(X,Y;\ZZ)$ and $H_2(X,Z;\ZZ)$ are free abelian groups generated by $B_{Y}$ and $B_{Z}$, respectively.

Let $V\approx \Si \times [0,1] \subset W$ be a collar of $\partial W=\Si$. The intersection of $\be_{i}$ and $\be'_{j}$ can be visualized in the thickened collar $V\times I$. For $i \ne j$, push $S_{i}$ and $S'_{j}$ straight down into $V\times I$ so that  $S_{i}$ lies lower in $V\times I$ than $S'_{j}$. At some $t \in I$, we see $S'_{j}$ in $V \times \{t\}$ together with a copy of $K_i$. The intersection of $\be_{i}$ and $\be'_{j}$ is the intersection of $K_i$ with $S'_{j}$, which is $\lk_{\Si}(K_i,K_j)$. For $i=j$, let $K_i^+$ be the longitude of $K_i$ obtained by pushing the core of $h_i^2$ off itself in the positive normal direction, i.e., $K_i^+=\partial(h_i^2)^+$. Arguing as above, it follows that $\cQ_X(\be_{i},\be'_{i})$ is the intersection of $K_i^+$ with $S'_{i}$. This intersection number is exactly $\fr(K_i)$. 
\end{proof}


\begin{figure}[h]
\centering
  \includegraphics[scale=.80]{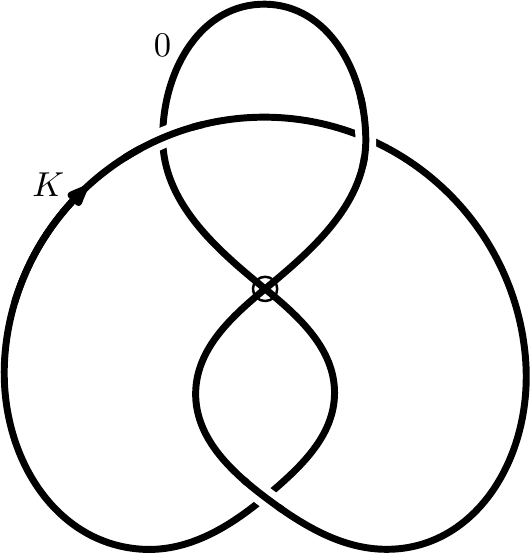}\hspace{1cm}
   \includegraphics[scale=.80]{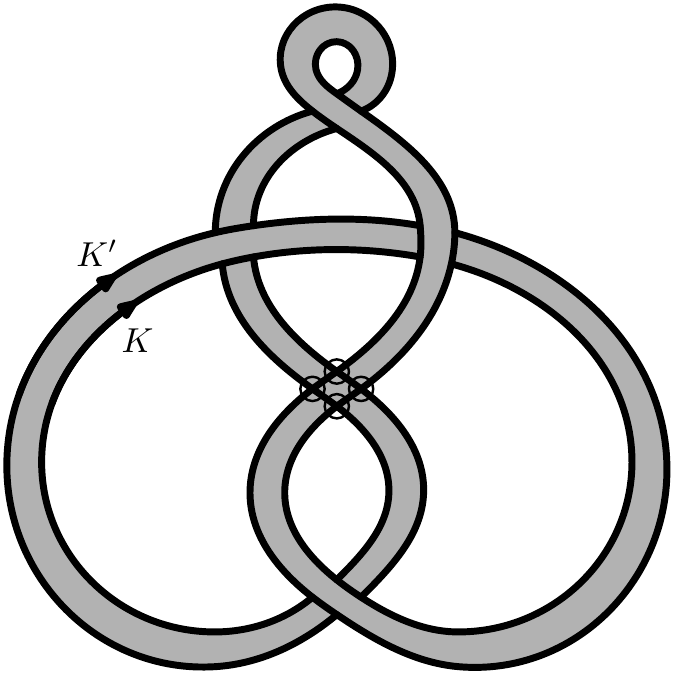} 
 \caption{A $0$-framed virtual knot (left) and a virtual ribbon that determines the framing (right).} \label{fig_framing}
\end{figure}


\subsection{Virtual linking matrices} \label{sec:vln}

We begin by recalling the notion of \textit{virtual linking numbers} (cf. Section \ref{sec:rel-link}). Given an oriented virtual link $L$ with components $J,K$, the virtual linking number is denoted $\vlk(J,K)$ and defined as the sum of the signs of the crossings where $J$ goes over $K$. One can check that the virtual linking numbers $\vlk(\cdot,\cdot)$ of an oriented virtual link coincide with the relative linking numbers $\lk_\Si(\cdot,\cdot)$ of the associated oriented link in a thickened surface.

For an individual component of $L$, its self-linking is specified by a choice of framing. In general, for a virtual knot $K$, a framing is a choice of parallel $K'$ to $K$. By \cite[Section 4.2]{Chrisman-2020}, a framing can be drawn as a \textit{virtual ribbon}, which is an immersed annulus in the plane with only virtual and classical band crossings as in Figure \ref{vc-band-crossing}. An example can be found in Figure \ref{fig_framing}.  

By convention, the $0$-framing of $K$ is the parallel with $\vlk(K,K')=\vlk(K',K)=0$. The $\ka$-framing is obtained from the $0$-framing by adding $|\ka|$ full twists, where we use right-hand twists if $\ka$ is positive and left-hand twists if $\ka$ is negative. As before, we use $\fr(K) \in \ZZ$ to denote a choice of framing for $K$. A \textit{framed oriented virtual link} $L=K_1 \cup \cdots \cup K_n$ is an oriented virtual link with a choice of framing for each component. The \textit{virtual linking matrix} of  $L$ is the $n \times n$ matrix with $(i,j)$ entry $\vlk(K_i,K_j)$ for $i \ne j$ and $\fr(K_i)$ for $i=j$. 

\begin{theorem} Every diagram of a framed oriented link on $\Si$ corresponds to a diagram of a framed oriented virtual link. Conversely, for every diagram of a framed oriented virtual link, there is a diagram of a framed oriented link on some closed oriented surface $\Si$. Under this correspondence, the virtual linking matrix is equal to the linking matrix appearing in Proposition \ref{prop_link_matrix_rel_int}.
\end{theorem}

\begin{proof} As is well known, every link in $\Si \times I$ corresponds to a virtual link. For a framed link in $\Si \times I$, we may replace the integer framing with a ribbon in $\Si \times I$. This gives a new link with twice as many components, which in turn corresponds to a virtual ribbon and a virtual link.  Conversely, given a framed virtual link, we may convert each framed component into a virtual ribbon. Using the construction shown in Figure \ref{vc-band-crossing}, we obtain a collection of ribbons in a thickened surface. Since the ribbons in $\Si \times I$ are mapped to virtual ribbons and vice versa, the framings are unchanged by the correspondence. The claim now follows from Proposition \ref{prop_link_matrix_rel_int}.
\end{proof}

\begin{remark}
For $4$-dimensional $2$-handlebodies, a handle slide alters the intersection form by a change of basis (see Section 5.1 in \cite{Gompf-Stipsicz}). As we shall see, the same is true for the relative intersection form of Proposition \ref{prop_link_matrix_rel_int}. 

Given a framed link $L \subset\Si \times I$ and
compact oriented 3-manifold $W$ with $\partial W =\Si$, let $X=(W \times I) \cup (h^2_1 \cup \cdots \cup h^2_n)$ be the $4$-manifold with 2-handles $h^2_1, \dots , h^2_n$ attached along $K_1 \cup \cdots \cup K_n$, the components of $L$. In sliding $h^2_i$ over $h_j^2$, the new handle has attaching sphere the framed knot $K_i^\star$, which is the band sum of $K_i$ and $K_j'$, where $K_j'$ is the parallel of $K_j$ given by the framing. Given an orientation of $L$, this operation is \textit{handle addition} if the band sum of $K_i$ and $K_j'$ respects their orientations and \textit{handle subtraction} otherwise. In the bases $B_Y$ for $H_2(X,Y)$ and $B_Z$ for $H_2(X,Z)$, the handle slide replaces $\beta_{i}$  with $\be_{i}\pm\be_{j}$ and $\be'_{i}$ with $\be'_{i}\pm\be'_{j}$, where the sign is ($+$) for handle addition and ($-$) for subtraction. The new component has framing 
\begin{align} \label{eqn_framing_virt}
\fr(K_i^\star) &=\fr(K_i)+\fr(K_j) \pm \lk_\Si(K_i,K_j) \pm \lk_\Si(K_j,K_i).
\end{align}
Therefore, the effect of a handle slide on linking matrix is to add (or subtract) row $j$ to row $i$ and column $j$ to column $i$.

All of this translates over to framed virtual links without loss of information.
Given a framed virtual link $L =K_1 \cup \cdots \cup K_n, $ a \textit{virtual handle slide} of $K_i$ over $K_j$ is defined by replacing $K_i$ by the band sum of $K_i$ and $K'_j$, where $K_j'$  is the parallel of $K_j$ given by its framing. The new component $K^\star_i$ has framing
\begin{align} \label{eqn_framing_virt}
\fr(K_i^\star) &=\fr(K_i)+\fr(K_j) \pm \vlk(K_i,K_j) \pm \vlk(K_j,K_i).
\end{align}
Just as for links in $\Sigma \times I$, the effect of a virtual handle slide on the virtual linking matrix is to add (or subtract) row $j$ to row $i$ and column $j$ to column $i$.
\end{remark}


\subsection{Virtual Kirby diagrams} 
In this subsection, we show how to realize the Gordon-Litherland pairing as the virtual linking matrix associated to a virtual Kirby diagram. By Theorem \ref{thm:isom}, this also realizes the relative intersection form of the mirror double branched cover. 

Recall that any spanning surface $F$ of a knot $K \subset \Si \times I$ can be deformed by an isotopy to a disk-band surface in $\Si \times I$. It consists of a single $0$-handle $F^0$ and a collection $F^1_1,\ldots,F_n^1$ of $1$-handles. Each $1$-handle $F_i$ is twisted by some number $\ka_i\in \ZZ$ of half twists, which are right-handed if $\ka_i\ge 0$ and left-handed if $\ka_i \le 0$. The disk-band surface may be drawn in the plane as a virtual disk-band surface with a single $0$-handle and $1$-handles. By abuse of notation, this will again be written as $F=F^0 \cup F_1^1 \cup \cdots\cup F_n^1$. It may be assumed that $F^0 \subset \RR^2$ is embedded in the lower half-plane $y \le 0$, that the $1$-handles all lie in the upper half-plane $y \ge 0$, and that the attaching spheres of all the $1$-handles lie on $y=0$.

A virtual link is now constructed from the virtual disk-band surface as follows. Rotate the configuration of $1$-handles in the upper half-plane around the $x$-axis in $\RR^3$ and place it in the lower half-plane $y \le 0$. Ignoring $F^0$, we see that there is a copy $V_i^1$ of $F_i^1$ in the lower half-plane. If $F_i^1$ crosses over $F_j^1$, then $V_j^1$ crosses over $V_i^1$. Every virtual crossing of bands appears as a virtual crossing of the corresponding bands in the lower half-plane. A half-twist in a band also appears the same in both half-planes. The union of the core $f_i$ of $F_i^1$ and $v_i$ of $V_i^1$ is a virtual knot diagram $K_i$. Set $\fr(K_i)=\ka_i$. 

Let $L_{F}$ be the framed virtual link $K_1 \cup \cdots \cup K_n$. It describes a relative handlebody structure for a 4-manifold, or rather family of 4-manifolds. Therefore, we regard $L_F$ as a \textit{virtual Kirby diagram.}

\begin{theorem} \label{thm:virtual_GL} The virtual linking matrix of $L_F$ is a matrix representing  the Gordon-Litherland pairing $\cG_F \co H_1(F;\ZZ) \times H_1(F;\ZZ) \to \ZZ$ and the relative intersection form $\cQ_F\co H_2(M_{\widehat{F}},Y;\ZZ) \times H_2(M_{\widehat{F}},Z;\ZZ) \to \ZZ$.  
\end{theorem} 
\begin{proof} Theorem \ref{thm:isom} implies that the second claim follows from the first. For the first, there is a basis $[\al_1],\ldots,[\al_n]$ of the first homology of $F=F^0 \cup F_1^1 \cup \cdots \cup F_n^1$, where $\al_i$ is a simple closed curve consisting of the core $f_i$ of $F_i^1$ and a simple path in $F^0$ connecting the ends of $f_i$. We first show that $\cG_F(\al_i,\al_i)$ is the framing coefficient of $K_i$:
$$\cG_F(\al_i,\al_i)  = \lk_{\Si}(\tau \al_i,\al_i)= \ka_i.$$
The double cover $\wt{F}$ of $F$ has two copies $(F^0)',(F^0)''$ of $F^0$ and two copies $(F_i^1)',(F_i^1)''$ of each $F_i^1$. One of the copies, say $(F^0)'$ lies slightly below $F^0$ in $\Si \times I$ and the other one, $(F^0)''$ lies slightly above.  For $i \ne j$, contributions to $\cG_F(\al_i,\al_j)$ come from band crossings between $F_i^1$ and $F_j^1$ and from the transverse intersections of $\al_i,\al_j$ on $F^0$. Let $x$ be an intersection point of $\al_i$ and $\al_j$. Lifting a small ball $B$ centered at $x$ in $F$ to $\wt{F}$ gives two small balls $B' \subset (F^0)',B'' \subset (F^0)''$ that are centered at the lifts of $x',x''$ of $x$. Observe that $\tau\al_i$ intersects each of $B'$ and $B''$ in an arc passing through the lifts of $x$. Only the arc passing through the higher sheet $B''$ contributes to $\lk_{\Si}(\tau\al_i,\al_j)$, and this contribution is precisely the local intersection number at $p(x) \in \Si$ of $p_*(\al_i)$ and $p_*(\al_j)$. Consequently, the contribution of any transverse intersection of $\al_i$ and $\al_j$ in $F^0$ to $\cG_F(\al_i,\al_j)$ is $0$.

Now, consider a band crossing of $F_i^1$ and $F_j^1$. Suppose that the cores $f_i,f_j$ have local crossing sign $\ep$ at this crossing. If $F_i^1$ crosses over $F_j^1$, then the contribution to $\cG_F(\al_i,\al_j)=\lk_{\Si}(\tau \al_i,\al_j)-p_*(\al_i)\cdot p_*(\al_j)$ is $2 \ep-\ep=\ep$. On the other hand, a band crossing of $F_j^1$ over $F_i^1$ contributes to $\cG_F(\al_i,\al_j)$ only in the term $-p_*(\al_i)\cdot p_*(\al_j)=p_*(\al_j)\cdot p_*(\al_i)$, and this contribution is again $\ep$. Thus, $\cG_F(\al_i,\al_j)$ is the sum of all the local crossings signs of band crossings between $F_i^1$ and $F_j^1$. By construction of $L_F$, this is exactly $\vlk(K_i,K_j)$ and the first claim follows immediately. 
\end{proof}

\begin{figure}[h]
\centering
 \includegraphics[scale=0.96]{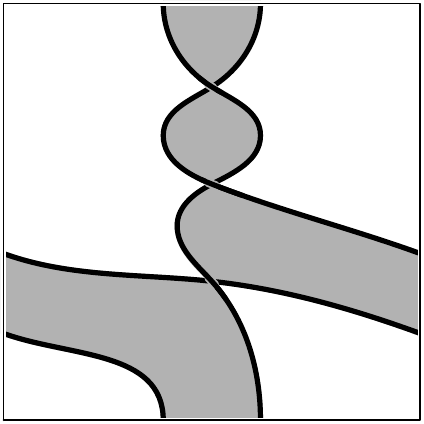}\hspace{.6cm}
  \includegraphics[scale=.52]{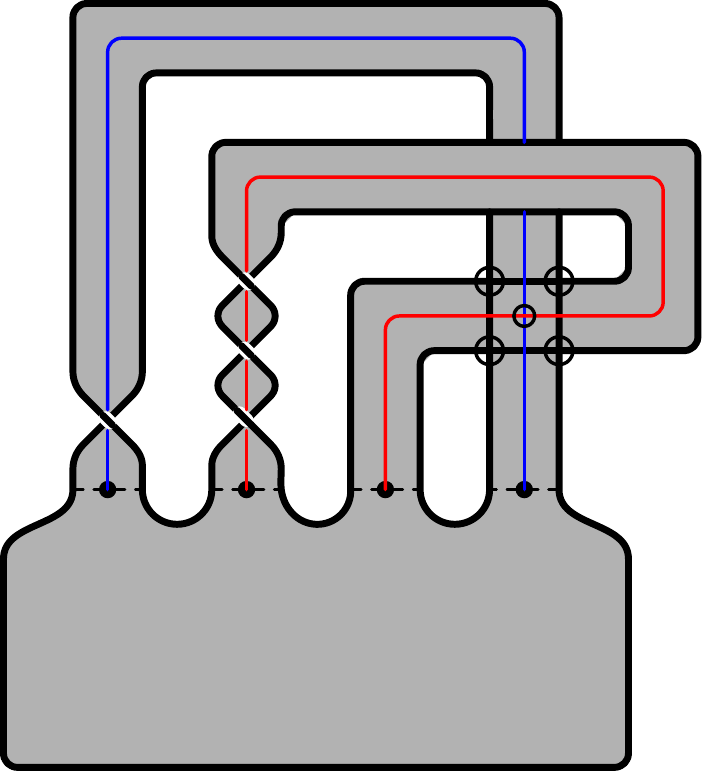}\hspace{.6cm}
   \includegraphics[scale=.90]{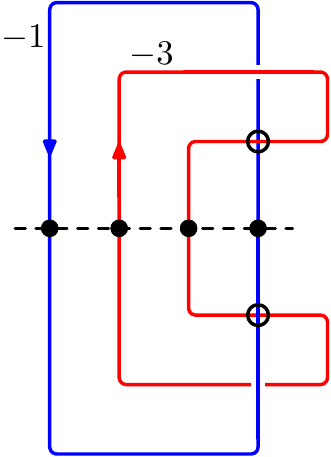} 
 \caption{Construction of the framed virtual link $L_F$ for a checkerboard surface $F$ of 3.7.} \label{fig_3_7_framed_link_1}
\end{figure}

\begin{figure}[h]
\centering
 \includegraphics[scale=0.96]{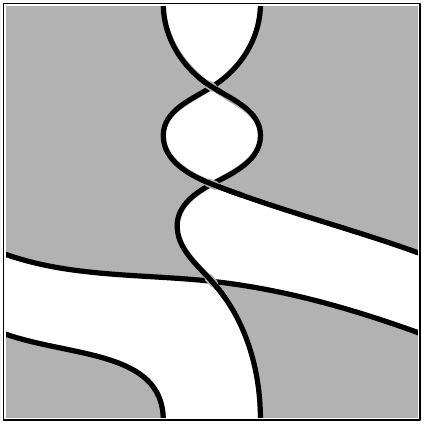}\hspace{.6cm}
  \includegraphics[scale=.52]{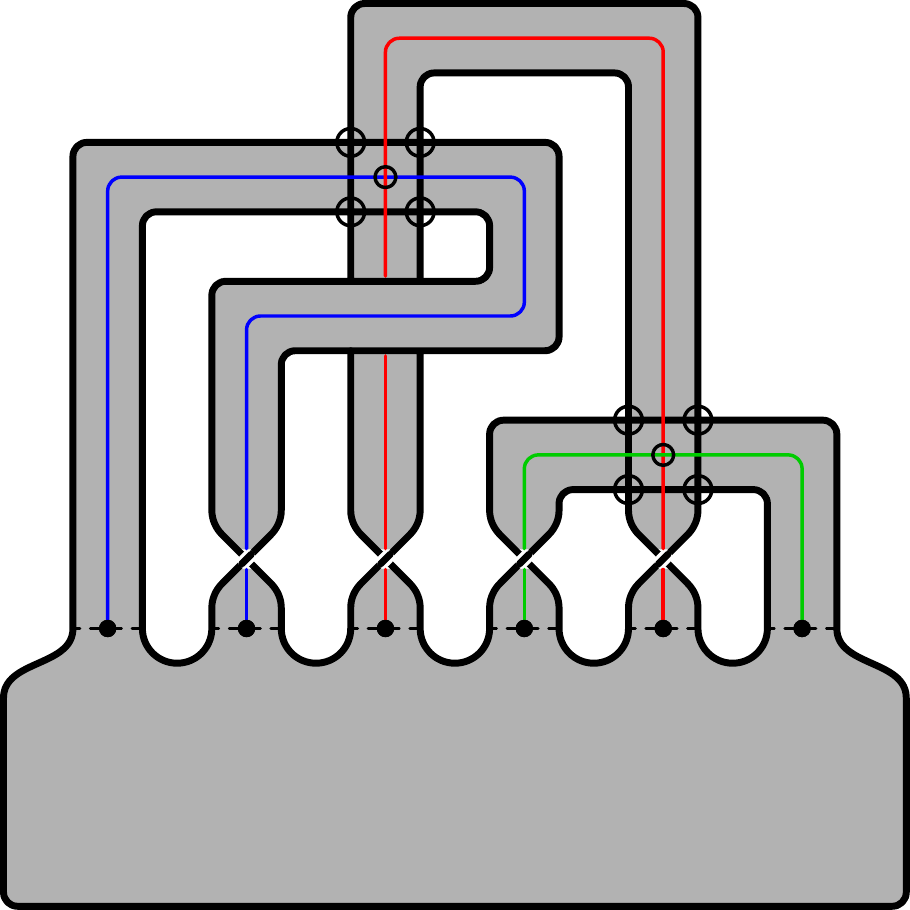}\hspace{.6cm}
   \includegraphics[scale=.90]{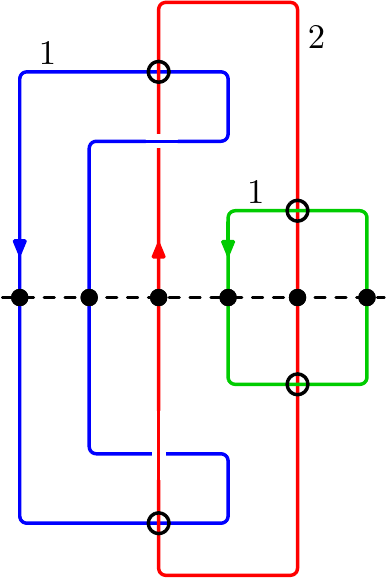} 
 \caption{Construction of the framed virtual link $L_{F^*}$ for the dual checkerboard surface $F^*$ of 3.7.} \label{fig_3_7_framed_link_2}
\end{figure}

\begin{example}\label{ex_framed_virtual_link_3_7} 
The construction of the framed virtual links for the two checkerboard surfaces $F$ and its dual $F^*$ of Example \ref{ex:3-7} is given in Figures \ref{fig_3_7_framed_link_1} and \ref{fig_3_7_framed_link_2}. Starting with the checkerboard surfaces (left), we isotope and convert them to virtual disk-band surfaces (middle). The framed virtual links $L_F$ and $L_{F^*}$ are drawn on the right.

 The virtual linking matrices of $L_F$ and $L_{F^*}$, respectively, are:
\[
\begin{bmatrix}
-1 & 1 \\ 1 & -3
\end{bmatrix}, 
\begin{bmatrix} 1 & 1 & 0 \\ 1 & 2 & 0 \\ 0 & 0 & 1 \end{bmatrix}.
\]
By Theorem \ref{thm:virtual_GL}, these matrices represent the relative intersection forms $\cQ_F$ and $\cQ_{F^*}$, respectively. Therefore, $\sig(\cQ_F) =-2$, and $\sig(\cQ_{F^*}) =3$. Note that this agrees with Example \ref{ex:3-7}, where we computed $\sig(\cG_{F})$ and $\sig(\cG_{F^*})$, and showed that $\si(K,F)=0,\ \det(K,F)= 2$ and $\si(K,F^*)= 2,\ \det(K,F^*)=1$. 
\hfill $\Diamond$ \end{example}

\begin{figure}[ht]
\centering
\includegraphics[scale=.5]{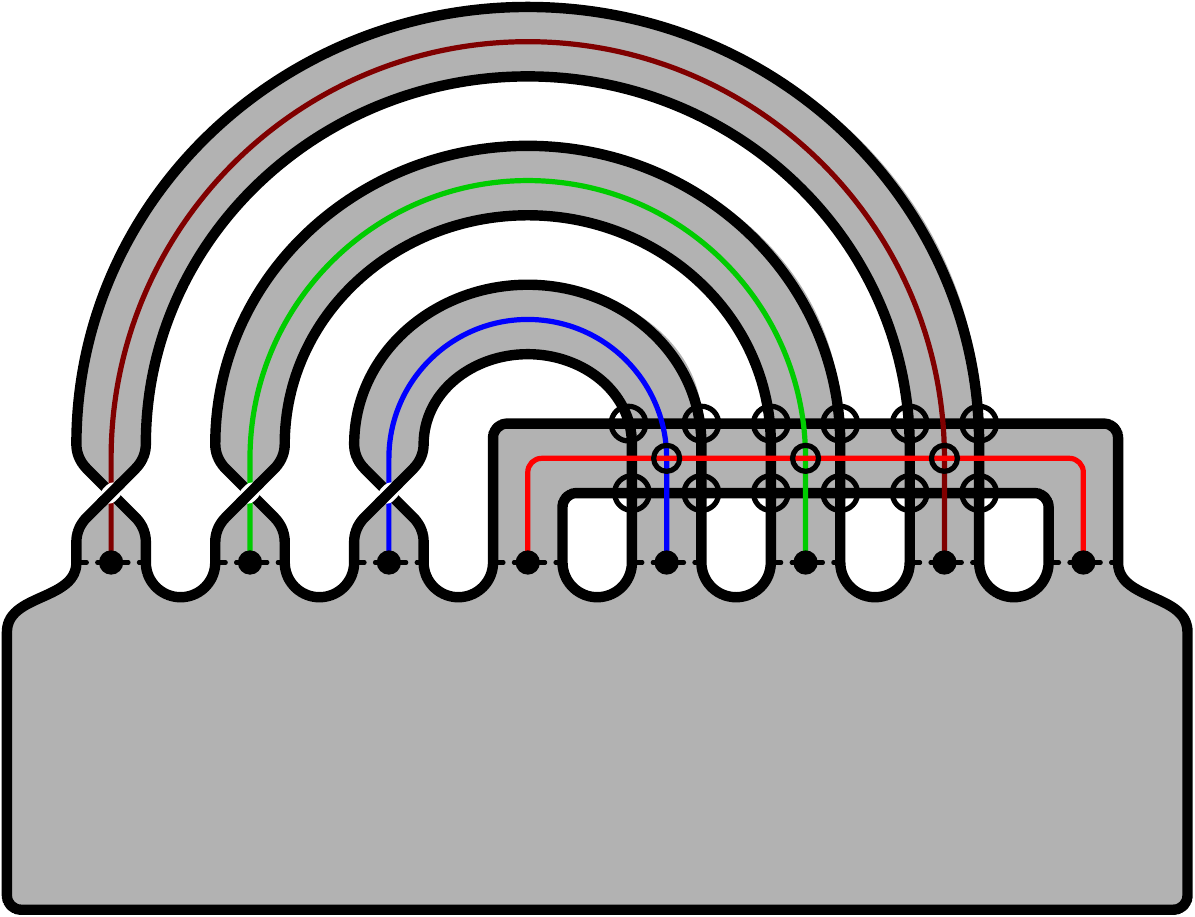}
\hspace{.2cm}
 \includegraphics[scale=1.4]{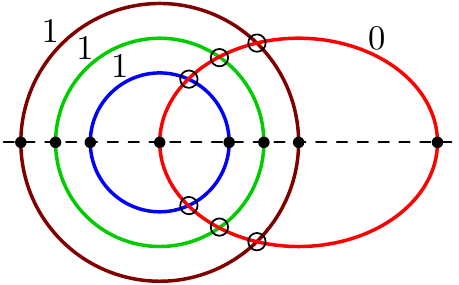}
\caption{The virtual disk-band surface $F^*$, and the virtual Kirby diagram $L_{F^*}$ for the non-minimal genus representative of the trefoil.} \label{non-minimal-trefoil-kirby}
\end{figure}

\begin{example}
For the trefoil knot with a diagram on the torus in Figure \ref{non-minimal-trefoil}, the disk-band surface and the framed virtual link $L_{F^*}$ are given in Figure \ref{non-minimal-trefoil-kirby}. The framed virtual link $L_F$ is the unknot with framing $-3$ (unpictured) and  virtual linking matrix $[-3]$. By Theorem \ref{thm:virtual_GL}, this matrix represents the relative intersection form $\cQ_F$, which has signature $\sig(\cQ_F) =-1.$

The virtual linking matrix of $L_{F^*}$ is given by
$$\begin{bmatrix} 1 & 0 & 0 &0 \\ 0&1 & 0 & 0 \\ 0 & 0 & 1&0\\ 0 & 0 & 0&0 \end{bmatrix}.$$
By Theorem \ref{thm:virtual_GL}, this matrix represents the relative intersection form $\cQ_{F^*}$,  which has signature $\sig(\cQ_{F^*}) =3.$ These computations agree with Example \ref{ex:non-minimal-trefoil}, where we computed $\sig(\cG_F)$ and $\sig(\cG_{F^*})$ and showed that $\si(K,F)=2, \det(K,F)=3$ and $\si(K,F^*)=3, \det(K,F^*) = 0$.
\hfill $\Diamond$ \end{example}

In \cite{Akbulut-Kirby-1980}, Akbulut and Kirby gave a method for constructing handlebody decompositions for branched covers over surfaces in $4$-manifolds. A minor adaptation of their method gives an alternate proof of Theorem \ref{thm:virtual_GL}, and we give a brief sketch. Push the interior of the spanning surface $F \subset \Si \times I$ slightly into the interior of $W \times I$. This may be assumed to occur in a collar $C \approx \partial (W \times I)\times I$, so that $F$ is pushed into $\Si \times I \times I$. Cutting out the trace of the isotopy that pushes $F$ into the interior of $W \times I$ gives $W \times I$ again, but with a thickened copy of $F$ visible in $\Si \times I$. Call this thickened copy $P$, where $P$ is the quotient space $(F \times [0,1])/[(x,t) \sim (x,t'), \forall t,t' \in I]$. As usual, we have a handle decomposition $F=F^0 \cup F_1^1 \cup \cdots \cup F_n^1$ as a disk-band surface. This gives a handle decomposition of $P$, $P \approx (F^0 \times I/\sim) \cup (F^1_1 \times I/\sim) \cup \cdots \cup (F_n^1 \times I/\sim)$. 

Recall that the space $M_{\widehat{F}}$ is constructed from two copies $W_1 \times I$, $W_2 \times I$ of $W \times I$ with the trace of the isotopy removed. We denote the copy of $F$ in $W_2 \times I$ by $V$ and  its handle decomposition by $V=V^0 \cup V_1^1 \cup \cdots \cup V_n^1$. The copy of $P$ in $W_2 \times I$ is denoted by $Q$. Then $M_{\widehat{F}}$ is constructed by identifying the corresponding handles of $P$ and $Q$. Observe that when the $0$-handles of $P$ and $Q$ are identified, the attaching spheres of $F_i^1$ and $V_i^1$ coincide. Let $K_i$ denote the union of the cores of $F_i^1$ and $V_i^1$. The handles $F_i^1 \times I/\sim$ and $V_i^1 \times I/\sim$ are identified by attaching a $2$-handle along $K_i$. The framing of $K_i$ is $\ka_i$, where $\ka_i \in \ZZ$ is the number of half-twists of the band $F_i^1$, with $\ka_i>0$ for right-hand twists and  $\ka_i<0$ for left-hand twists. 

Now, let $h_1^2,\ldots,h_n^2$ be the $2$-handles attached along the knots $\al_1,\ldots,\al_n$. As in the proof of Proposition \ref{prop_link_matrix_rel_int}, the cores of $h_1^2,\ldots,h_n^2$ correspond to generators of $H_2(M_{\widehat{F}},X;\ZZ)$ and $H_2(M_{\widehat{F}},Y;\ZZ)$. Furthermore, the linking matrix of the framed link $\al_1 \cup \cdots \cup \al_n$ is the matrix of the relative intersection form in this basis. By construction of the handle decomposition of $M_{\widehat{F}}$, this is identical to the virtual linking matrix of $L_F$.

\subsection*{Concluding remarks} The methods developed in this paper have a number of applications, and here we mention a few of them. For instance, the Gordon-Litherland pairing has already been used to give a topological characterization of alternating links in thickened surfaces in \cite{Boden-Karimi-2020}. 
One can also use the tools developed in this paper to profitably explore the crosscap numbers and the nonorientable 4-genus of virtual knots. In the recent paper \cite{Boden-Karimi-2021}, we investigate the concordance properties of signature invariants of checkerboard colorable knots in thickened surfaces.

\subsection*{Acknowledgements}
The first author was partially funded by the Natural Sciences and Engineering Research Council of Canada, and the second author was partially funded by The Ohio State University, Marion Campus. We would like to thank Robin Gaudreau, Aaron Kaestner, Andy Nicas, Danny Ruberman, Will Rushworth, Robb Todd, and Lindsay White for their feedback and input.

\newpage

\end{document}